\documentclass[a4paper,12pt, reqno]{amsart}
\usepackage{amsmath, amssymb, amsfonts, multirow, multicol, enumitem, mathtools, algpseudocode, mathrsfs,  comment, xspace, diagbox, setspace, natbib}
\usepackage[foot]{amsaddr}

\usepackage[colorlinks]{hyperref}
\setcitestyle{authoryear, round}
\hypersetup{linkcolor=black, urlcolor=black, citecolor=black}
\usepackage[margin=3cm]{geometry}

\usepackage{tikz}
\usetikzlibrary{matrix, chains, arrows, knots}

\newcommand{\uV}{\reflectbox{\rotatebox[origin=c]{180}{$V$}}}

\newcommand{\N}{\mathbb{N}}

\newcommand{\n}{\mathbf n}

\newcommand{\m}{\mathbf{m}}

\newtheorem{thm}{Theorem}[section]
\newtheorem{lem}[thm]{Lemma}

\theoremstyle{definition}
\newtheorem{definition}[thm]{Definition}

\newtheorem{problem}[thm]{Problem}

\DeclareMathOperator{\im}{im}
\DeclareMathOperator{\dom}{dom}

\newcommand{\caseseven}{\smash{\left\{\vphantom{\begin{aligned} \\ \\ \\ \\ \\ \\ \\[1ex]\end{aligned}}\right.}\!}
\newcommand{\phantomseven}{\mathrel{\phantom{=}}\hphantom{\caseseven}} 
\newcommand{\casetwo}{\smash{\left\{\vphantom{\begin{aligned} \\ \\[-0.75ex] \end{aligned}}\right.}\!}
\newcommand{\phantomtwo}{\mathrel{\phantom{=}}\hphantom{\casetwo}}
\newcommand{\lefttermtwo}[1]{\smash{\raisebox{-0.8em}{$#1$}} &\smash{\raisebox{-0.8em}{ $= \casetwo$}}}

\newcommand{\phantomtwoC}{\mathrel{\phantom{=}}\hphantom{\casetwo}}
\newcommand{\lefttermtwoC}[1]{\smash{\raisebox{-0.7em}{$#1$}} &\smash{\raisebox{-0.7em}{ $= \casetwo$}}}

\newcommand{\ltag}[1]{\tag*{#1}\label{#1}} %label and tag
\title[An algebraic inversion and deletion model]{An algebraic model for inversion and deletion in bacterial genome rearrangement}
\author{Chad Clark\textsuperscript{\,\lowercase{a$\ast$}}}
\address[a]{Centre for Research in Mathematics and Data Science, Western Sydney University, Penrith, NSW, Australia}
\email{\textsuperscript{*}Corresponding author: chad.clark@westernsydney.edu.au}
\author{Julius Jonu\v sas\textsuperscript{\,\lowercase{b}}}
\address[b]{Mathematical Institute, School of Mathematics and Statistics, University of St Andrews, St Andrews, UK}
\author{James D. Mitchell\textsuperscript{\,\lowercase{b}}} 
\author{Andrew Francis\textsuperscript{\,\lowercase{a}}}
\begin{document}

\begin{abstract}
Inversions, also sometimes called reversals, are a major contributor to variation among bacterial genomes, with studies suggesting that those involving small numbers of regions are more likely than larger inversions. Deletions may arise in bacterial genomes through the same biological mechanism as inversions, and hence a model that incorporates both is desirable.  However, while inversion distances between genomes have been well studied, there has yet to be a model which accounts for the combination of both deletions and inversions. 

To account for both of these operations, we introduce an algebraic model that utilises partial permutations.  This leads to an algorithm for calculating the minimum distance to the most recent common ancestor of two bacterial genomes evolving by inversions (of adjacent regions) and deletions. 
The algebraic model makes the existing short inversion models more complete and realistic by including deletions, and also introduces new algebraic tools into evolutionary distance problems.


\end{abstract}
\maketitle

\section{Introduction} \label{sec:intro}
Methods for computing the evolutionary distance between bacterial genomes are important for phylogenetic reconstruction, especially by way of contrast with organisms that have morphological characteristics and better defined species boundaries.  Approaches to distances based on large-scale rearrangements have been widely studied in bacteria because they are often relatively quick to compute and can be used to complement, or even improve, trees based on other methods such as sequence comparisons \citep{bochkareva2018genome}.

The bacterial genomes that we will consider have a single circular chromosome.  During the evolution of bacterial genomes a frequent rearrangement event is the inversion, where the clockwise order of a contiguous block of conserved regions is reversed \citep{eisen2000evidence}. If the orientation of regions is taken into account, these events also reverse the orientations of regions in this block. While most early mathematical models assumed the probability of all inversions to be equal, evidence to the contrary has emerged which suggests shorter inversions are more likely~\citep{seoighe2000prevalence, Dalevi2002measuring, Lefebvre2003detection, darling2008}. With this in mind, throughout this paper we will be concerned with inversions of length two. 

Many other large scale changes to bacterial DNA have been observed and investigated, notably insertion of novel DNA (horizontal gene transfer), deletion of segments, translocation of segments to different locations on the genome, and duplication of segments \citep{saier2008bacterial}. Deletions are special in the context of inversions however, because they can occur by the same mechanism, namely site-specific recombination~\citep{plasterk}. This means that inversions and deletion are related biologically in a way that other combinations of rearrangement operations are not.

Site-specific recombination acts on the circular genome by forming a synaptic complex around two copies of a specific sequence on the genome, that might be far apart on the sequence but close together in a three-dimensional sense in the cell.  The recombinase then cuts the DNA at both sites and rejoins across the two, in effect locally replacing a trivial 2-braid with a braid generator (as an algebraic topologist might describe it).  This event can result in the inversion of a segment of the genome relative to the rest of the genome, but can also result in the deletion of a segment, as shown in Figure~\ref{f:ssr}.  

\begin{figure}[ht]
\centering
\begin{tikzpicture}
\begin{knot}[
% draft mode = crossings, 
clip width = 5, 
consider self intersections ,
% ignore endpoint intersections=false ,
flip crossing = 2
]
\shade[left color= gray,right color= gray] (.8,-2) rectangle (2.2,-3 ); %left color=blue,right color=red
\strand [ultra thick] (0,-2) 
	to [out=up,in=up] (1,-2)
	to (1,-3)
	to [out=down,in=up] (0,-4.4)
	to [out=down, in=left] (1,-5)
	to [out=right,in=south west] (2,-4.4)
	to [out=north east,in=down] (3,-3)
	to (3,-2)
	to [out=up,in=up] (2,-2)
	to (2,-3)
	to [out=down,in=up] (3,-4.4)
	to [out=down,in=right] (2,-5)
	to [out=left,in=south east] (1,-4.4)
	to [out=north west,in=down] (0,-3)
	to (0,-2);
\end{knot}
\draw[->,>=latex,thick](3.5,-3.5)--(4.5,-3.5);
\end{tikzpicture}
\hspace{2mm}
\begin{tikzpicture}
\begin{knot}[
% draft mode = crossings, 
clip width = 5,
consider self intersections ,
% ignore endpoint intersections=false ,
]
\strand [ultra thick] (0,-2) 
	to [out=up,in=up] (1,-2)
	to [out=down,in=up] (2,-3)
	to [out=down,in=up] (3,-4.4)
	to [out=down,in=right] (2,-5)
	to [out=left,in=south east] (1,-4.4)
	to [out=north west,in=down] (0,-3)
	to (0,-2);

\strand [ultra thick] (1,-3)
	to [out=down,in=up] (0,-4.4)
	to [out=down, in=left] (1,-5)
	to [out=right,in=south west] (2,-4.4)
	to [out=north east,in=down] (3,-3)
	to (3,-2)
	to [out=up,in=up] (2,-2)
	to [out=down,in=up] (1,-3);
\flipcrossings{1,2,4}
\end{knot}
\end{tikzpicture}
\caption{Site-specific recombination giving rise to deletion, with the area of recombinase action shown shaded on the left.  The result is in fact a pair of linked components (topologically, a ``Hopf link''), but over time any component without the essential genes from the original genome (such as origin and terminus of replication) would degrade and the result would be a genome without the genetic material from that component (that is, a deletion).  If the figure on the left had an even number of twists the result would be an inversion (see for example~\citet{francis2014algebraic}).}
\label{f:ssr}
\end{figure}
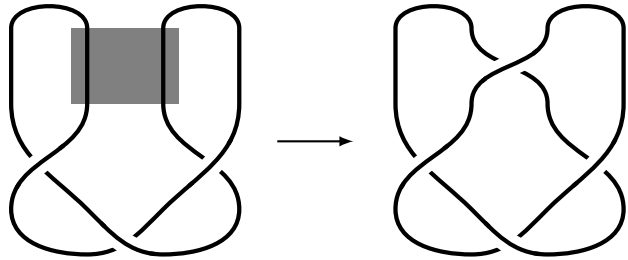

Most rearrangement models, with few exceptions (see \citet{alexandrino2021} for instance), assume that the genomes in question have the same sets of regions. While inversion models need both genomes to have the same  gene content (or ignore gene content that is not shared), a model incorporating both inversions and deletions can model an evolutionary history of two genomes with differing gene content under the assumption that they both evolved from a common ancestor with the union of their sets of genes.  Incorporating deletions thus enables a wider class of genomes to be compared more completely, especially since in some instances (see \citet{raeside2014large}) deletions are the most frequently observed recombination event. 

By thinking of bacterial genomes as sequences of region labels or integers (see \citet{bhatia2018position} for a review of these conventions), a pair of genomes $\sigma_1$ and $\sigma_2$ can be represented by signed or unsigned permutations, assuming all regions are distinct.  The minimum length sequence of operations $t_1,\cdots , t_k$ such that $\sigma_1t_1\cdots t_k = \sigma_2$ consequently provides an estimate of the evolutionary distance between these genomes.  These distances may then be used to reconstruct phylogenetic trees using methods such as neighbor-joining \citep{Saitou1987neighbour}. 

Although finding the unsigned inversion distance between two genomes is \textbf{NP}-hard \citep{caprara1997sorting}, the signed inversion distance can be found in polynomial time when all inversions (of any length) are assumed to be equally likely \citep{hannenhalli1999transforming}. For unsigned inversions, an upper bound on the inversion distance between genomes was first provided in \citep{Watterson-chrom-reversal-1982}, with polynomial time algorithms later established from a combinatorial perspective by \citet{Jerrum1985complexity} and an algebraic perspective by \citep{egrinagy2014group}. Polynomial time algorithms also exist for signed inversion distances \citep{galvao2017sorting, oliveira2018sorting} (using terms such as ``super short reversal''). 

When a polynomial time algorithm for a rearrangement distance exists, it is often possible to incorporate both deletions and \emph{insertions} into the model.  Polynomial time algorithms exist for calculating the minimal genomic distance under exclusively insertions and deletions \citep{marron2004genomic}, with insertions, deletions and signed inversions \citep{el2000genome}, and with inversions, transpositions, insertions and deletions \citep{alexandrino2021genome}. Insertions and deletions have also been incorporated into other models such as double cut and join \citep{braga2010genomic, shao2012approximating}. 

When insertions and deletions are both allowed, the minimum distance between any pair of genomes $G_1$ and $G_2$ with region labels $R_1$ and $R_2$ respectively always exists. Furthermore, this distance is symmetric in the sense that the distance from $G_1$ to $G_2$ is the same as the distance from $G_2$ to $G_1$, because the deletion of a region can be ``undone'' by inserting the deleted region back into the genome and vice versa. There is, however, little work that considers the addition of deletions without also considering insertions. When considering deletions without insertions, unless we make the assumption that $R_1 \subseteq R_2$ or $R_2 \subseteq R_1$ or both (as in~\citet{el2000genome}), there will not necessarily be an inversion/deletion sequence that transforms one genome into the other. To deal with this, we will provide a model for directly reconstructing the most recent common ancestor of $G_1$ and $G_2$. 

This model will make use of partial analogues of the symmetric group, namely the symmetric inverse monoid and the symmetric inverse category, which will be discussed in Section \ref{s:alg.prelim}. To the best of the authors' knowledge they have not yet been explicitly used in any distance-based methods.  
When working with these structures, it is advantageous to adopt the convention of writing maps on the right and composing from left to right. That is, we write $(x)f$ instead of $f(x)$, and $fg$ is written instead of $g \circ f$.  

Hereafter we will use the term ``inversion'' to mean an inversion of precisely two adjacent regions. The paper proceeds as follows. In Section \ref{sec:algmodel} we provide an algebraic framework for describing bacterial genomes. After introducing a number of key algebraic structures in Section \ref{s:alg.prelim}, these structures are then used to establish an algebraic model of the inversion/deletion process in Section \ref{sec:invdelmodel}. This allows us to define a problem called the region alignment problem, where it will be shown that solving this problem over all pairs of orientations of $G_1$ and $G_2$ allows for the reconstruction of a parsimonious most recent common ancestor with respect to inversion and deletions. An exact algorithm for calculating this distance is provided in Section \ref{sec:exact}.  The paper ends with a Discussion in Section~\ref{s:discussion} that describes some of the important limitations of the models here, and also some of the opportunities for further development.  In particular, it is to be hoped that the introduction of the semigroup models here will lead to further work by algebraists to improve applicability and utility of genome rearrangement models.

\section{An Algebraic Model of Bacterial Genomes}\label{sec:algmodel}

For a circular genome $G$ with a set $R$ of $n$ distinct regions, different rotations and reflections of $G$ represent different ways of viewing the genome in three dimensional space. These symmetries are accounted for by an action of the dihedral group $D_{n}$, which consists of permutations in the symmetric group $S_n$ (the group of permutations of $\mathbf{n} = \{1, \dots, n\}$) representing the rotations and reflections of an $n$-gon. Beginning with a set $X_R$ containing the $n!$ words of length $n$ whose distinct letters are from $R$, consider the action $\cdot$ of $D_n$ on $X_R$ where for $\sigma \in D_n$ we have
\[
\sigma \mathrel{\cdot} x_1\cdots x_n = x_{(1)\sigma}\cdots x_{(n)\sigma}. 
\]
The equivalence relation $\sim$ on $X_R$ induced by this action (where words $u, v \in X_R$ are related if and only if there exists $\sigma \in D_n$ such that $u = \sigma \mathrel{\cdot} v$) allows for the following algebraic definition of a circular genome.
\begin{definition}
A \emph{genome} $G$ with region set $R$ is an equivalence class in the quotient set $X_R/\sim$. 
\end{definition}
For $u \in X_R$ the equivalence class of $u$ is denoted by $[u]$, elements of each equivalence class (words in $X_R$) are called the \emph{reference frames} of $G$, and for two genomes $G_1$ and $G_2$ a \emph{reference pair} is an element of the Cartesian product $G_1 \times G_2$.

To visualise the reference frames of a genome, begin with the unit circle centered at $(0,0)$ in $\mathbb{R}^2$ and specify a distinguished point at $(0,1)$. Subdivide the circle into $n$ arcs of equivalent length proceeding clockwise from $(0,1)$ where the arc immediately clockwise from $(0,1)$ is considered to be position 1, the next arc clockwise is considered to be position $2$ and so on until we reach position $n$ (which will be the arc directly anti-clockwise from $(0,1)$). If $x_1\cdots x_n$ is a reference frame of $G$ then its diagram is obtained by labelling position $i$ by $x_i \in R$ via bijection $\lambda : R \to \n$ from regions to positions (see Figure \ref{fig:orientation}). With this is mind, these bijections may also be used to represent reference frames rather than elements of $X_R$. 

\begin{figure}[ht]
\begin{center}
\begin{tikzpicture}[scale = 0.5]

      \draw (0,0) circle (2cm);
      \draw (0,0) node {$g_1$};
      \foreach \x in {0,45,90,...,315} \draw (\x:1.9cm) -- (\x:2.1cm);
      \def\rad{2.35cm}
      \draw (67.5:\rad) node[] {\footnotesize $a$}; %subtract 45 degrees going clockwise around.
      \draw (22.5:\rad) node[] {\footnotesize $b$};
      \draw (-22.5:\rad) node[] {\footnotesize $c$};  
      \draw (-67.5:\rad) node[] {\footnotesize $d$};
      \draw (-112.5:\rad) node[] {\footnotesize $e$};  
      \draw (-157.5:\rad) node[] {\footnotesize $f$};
      \draw (-202.5:\rad) node[] {\footnotesize $g$};  
      \draw (-247.5:\rad) node[] {\footnotesize $h$};  
      
\begin{scope}[xshift=8cm]
      \draw (0,0) circle (2cm);
      \draw (0,0) node {$g_2$};
      \foreach \x in {0,45,90,...,315} \draw (\x:1.9cm) -- (\x:2.1cm);
      \def\rad{2.35cm}
      \draw (67.5:\rad) node[] {\footnotesize $c$}; %subtract 45 degrees going clockwise around.
      \draw (22.5:\rad) node[] {\footnotesize $d$};
      \draw (-22.5:\rad) node[] {\footnotesize $e$};  
      \draw (-67.5:\rad) node[] {\footnotesize $f$};
      \draw (-112.5:\rad) node[] {\footnotesize $g$};  
      \draw (-157.5:\rad) node[] {\footnotesize $h$};
      \draw (-202.5:\rad) node[] {\footnotesize $a$};
      \draw (-247.5:\rad) node[] {\footnotesize $b$};  
      
\end{scope}

\begin{scope}[xshift=16cm]
      \draw (0,0) circle (2cm);
      \draw (0,0) node {$g_3$};
      \foreach \x in {0,45,90,...,315} \draw (\x:1.9cm) -- (\x:2.1cm);
      \def\rad{2.35cm}
      \draw (67.5:\rad) node[] {\footnotesize $h$}; %subtract 45 degrees going clockwise around.
      \draw (22.5:\rad) node[] {\footnotesize $g$};
      \draw (-22.5:\rad) node[] {\footnotesize $f$};  
      \draw (-67.5:\rad) node[] {\footnotesize $e$};
      \draw (-112.5:\rad) node[] {\footnotesize $d$};  
      \draw (-157.5:\rad) node[] {\footnotesize $c$};
      \draw (-202.5:\rad) node[] {\footnotesize $b$};
      \draw (-247.5:\rad) node[] {\footnotesize $a$};  
      
\end{scope}
\end{tikzpicture}
\end{center}
\caption{Given a set of regions $R = \{a,b,c,d,e,f,g,h\}$, the reference frames $g_1 = abcdefgh, g_2 = cdefghab$ and $g_3 = hgfedcba$ of the genome $[abcdefgh]$ represent different ways of viewing the same circular genome in three dimensional space. The reference frame $g_2$ is obtained by rotating $g_1$ two positions anticlockwise and $g_3$ is obtained by reflecting $g_1$ in the vertical axis.}
\label{fig:orientation}
\end{figure}
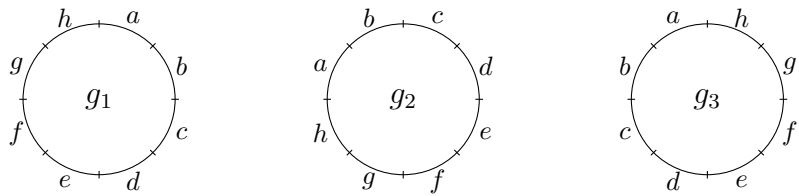 

We will proceed under the assumption that each genome has arisen via the minimum possible number of inversions and deletions, which is commonly known as the \emph{parsimony criterion}. This approach allows genome rearrangement problems to be viewed as combinatorial optimisation problems whose minimised solutions represent evolutionary distances in accordance with this criterion \citep{fertin2009combinatorics}. With this assumption in mind the most recent common ancestor of genomes $G_1$ and $G_2$ with region sets $R_1$ and $R_2$ respectively will have region set $R_1\cup R_2$, noting that it must certainly contain the union of the two sets of regions, but could possibly contain more (in which case a greater number of deletions would be required to yield $G_1$ and $G_2$, contradicting the parsimony criterion). 

Figure \ref{fig:example} illustrates an example of how reference frames $g_1$ and $g_2$ of genomes $G_1$ and $G_2$ respectively may arise via inversions and deletions from a (not necessarily most recent) common ancestor $A$. 

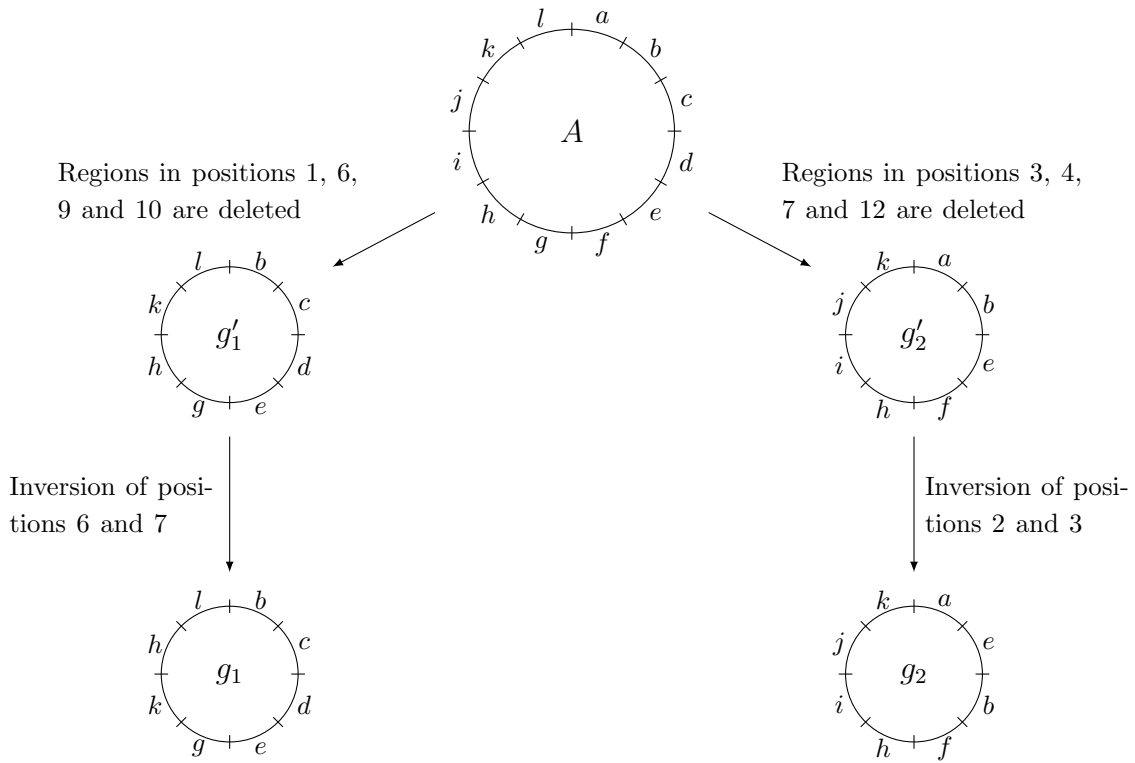
\begin{figure}[h]
\begin{center}
\begin{tikzpicture}[scale = 0.9]
      \draw (0,0) circle (1.5cm);
      \draw (0,0) node {$A$};
      \foreach \x in {0,30,60,...,330} \draw (\x:1.4cm) -- (\x:1.6cm);
      \def\rad{1.73cm}
      \draw (75:\rad) node[] {\footnotesize $a$};
      \draw (45:\rad) node[] {\footnotesize $b$};
      \draw (15:\rad) node[] {\footnotesize $c$};
      \draw (-15:\rad) node[] {\footnotesize $d$};
      \draw (-45:\rad) node[] {\footnotesize $e$};
      \draw (-75:\rad) node[] {\footnotesize $f$};
      \draw (-105:\rad) node[] {\footnotesize $g$};
      \draw (-135:\rad) node[] {\footnotesize $h$};
      \draw (-165:\rad) node[] {\footnotesize $i$};
      \draw (-195:\rad) node[] {\footnotesize $j$};
      \draw (-225:\rad) node[] {\footnotesize $k$};
      \draw (-255:\rad) node[] {\footnotesize $l$};
      \draw[->,>=latex](2,-1.2)-- node[above right =2mm, text width = 4cm] {\footnotesize{Regions in positions 3, 4, 7 and 12 are deleted}} (3.5,-2);
      \draw[->,>=latex](-2,-1.2)-- node[above left=2mm, text width = 4cm] {\footnotesize{Regions in positions 1, 6, 9 and 10 are deleted}} (-3.5,-2);
\begin{scope}[xshift=5cm,yshift=-3cm]
      \draw (0,0) circle (1cm);
      \draw (0,0) node {$g_2^\prime$};
      \foreach \x in {0,45,90,...,315} \draw (\x:.9cm) -- (\x:1.1cm);
      \def\rad{1.18cm}
      \draw (67.5:\rad) node[] {\footnotesize $a$}; %subtract 45 degrees going clockwise around.
      \draw (22.5:\rad) node[] {\footnotesize $b$};
      \draw (-22.5:\rad) node[] {\footnotesize $e$};  
      \draw (-67.5:\rad) node[] {\footnotesize $f$};
      \draw (-112.5:\rad) node[] {\footnotesize $h$};  
      \draw (-157.5:\rad) node[] {\footnotesize $i$};
      \draw (-202.5:\rad) node[] {\footnotesize $j$};  
      \draw (-247.5:\rad) node[] {\footnotesize $k$};  
      \draw[->,>=latex](0,-1.5)-- node[right, text width = 3cm] {\footnotesize{Inversion of positions 2 and 3}} (0,-3.5);
\end{scope}
\begin{scope}[xshift=-5cm,yshift=-3cm]
      \draw (0,0) circle (1cm);
\draw (0,0) node {$g_1^\prime$};
      \foreach \x in {0,45,90,...,315} \draw (\x:.9cm) -- (\x:1.1cm);
      \def\rad{1.18cm}
      \draw (67.5:\rad) node[] {\footnotesize $b$}; %subtract 45 degrees going clockwise around.
      \draw (22.5:\rad) node[] {\footnotesize $c$};
      \draw (-22.5:\rad) node[] {\footnotesize $d$};  
      \draw (-67.5:\rad) node[] {\footnotesize $e$};
      \draw (-112.5:\rad) node[] {\footnotesize $g$};  
      \draw (-157.5:\rad) node[] {\footnotesize $h$};
      \draw (-202.5:\rad) node[] {\footnotesize $k$};  
      \draw (-247.5:\rad) node[] {\footnotesize $l$};  
      \draw[->,>=latex](0,-1.5)-- node[left = -0.25cm, text width = 3cm] {\footnotesize{Inversion of positions 6 and 7}} (0,-3.5);

\end{scope}
\begin{scope}[xshift=5cm,yshift=-8cm]
      \draw (0,0) circle (1cm);
      \draw (0,0) node {$g_2$};
      \foreach \x in {0,45,90,...,315} \draw (\x:.9cm) -- (\x:1.1cm);
      \def\rad{1.18cm}
      \draw (67.5:\rad) node[] {\footnotesize $a$}; %subtract 45 degrees going clockwise around.
      \draw (22.5:\rad) node[] {\footnotesize $e$};
      \draw (-22.5:\rad) node[] {\footnotesize $b$};  
      \draw (-67.5:\rad) node[] {\footnotesize $f$};
      \draw (-112.5:\rad) node[] {\footnotesize $h$};  
      \draw (-157.5:\rad) node[] {\footnotesize $i$};
      \draw (-202.5:\rad) node[] {\footnotesize $j$};  
      \draw (-247.5:\rad) node[] {\footnotesize $k$};  
\end{scope}
\begin{scope}[xshift=-5cm,yshift=-8cm]
      \draw (0,0) circle (1cm);
      \draw (0,0) node {$g_1$};
      \foreach \x in {0,45,90,...,315} \draw (\x:.9cm) -- (\x:1.1cm);
      \def\rad{1.18cm}
      \draw (67.5:\rad) node[] {\footnotesize $b$}; %subtract 45 degrees going clockwise around.
      \draw (22.5:\rad) node[] {\footnotesize $c$};
      \draw (-22.5:\rad) node[] {\footnotesize $d$};  
      \draw (-67.5:\rad) node[] {\footnotesize $e$};
      \draw (-112.5:\rad) node[] {\footnotesize $g$};  
      \draw (-157.5:\rad) node[] {\footnotesize $k$};
      \draw (-202.5:\rad) node[] {\footnotesize $h$};  
      \draw (-247.5:\rad) node[] {\footnotesize $l$};  

\end{scope}

\end{tikzpicture}
\end{center}
\caption{An example of frames of reference $g_1$ and $g_2$ of $G_1$ and $G_2$ arising from an ancestor $A$ with the deletions occurring first, followed by inversions. After the deletions but prior to inversions there are intermediate reference frames $g_1^\prime$ and $g_2^\prime$ of genomes $G_1^\prime$ and $G_2^\prime$.}
\label{fig:example}
\end{figure}

\section{The Symmetric Group, the Symmetric Inverse Monoid and their Generalisations}\label{s:alg.prelim}

To model the inversion/deletion process and formalise the notion of a distance between genomes we use the machinery of the symmetric group, the symmetric inverse monoid and their generalisations. Throughout we let $\mathbf{n} = \{1, \dots, n\}$ for all positive integers $n$ (where $\mathbf{0} = \varnothing$),  let $\N = \{0, 1, \dots \}$ and $\N^+ = \N \setminus \{0\}$, and let the restriction of a map $f$ to a subset $X$ of its domain be denoted by $f|_X$. 

\begin{definition}\label{d:partial.perm}
Let $X$, $Y$ and $X^{\prime}$ be sets where $X^{\prime} \subseteq X$. A \emph{partial permutation} with domain $X^{\prime}$ from $X$ to $Y$ is an injection $f|_{X^{\prime}} : X^{\prime} \to Y$ where $(x)f= (x)f|_{X^\prime}$ for all $x \in X^\prime$ and where $(x)f$ is undefined for all $x\in X\setminus X^\prime$. The \emph{domain} of $f$ is denoted $\dom(f)$, while $\im(f)$ is the \emph{image} of $X^{\prime}$ under $f|_{X^{\prime}}$.
\end{definition}

For a monoid $M$ the \emph{inverse} of $m\in M$ is the unique $m^{-1} \in M$ such that $mm^{-1}m = m$ and $m^{-1}mm^{-1} = m^{-1}$. If all elements of $M$ have an inverse in this sense, then $M$ is an \emph{inverse monoid}. The set of partial permutations from $\mathbf{n}$ to itself, which is denoted $\mathcal{I}_{n}$, is an inverse monoid called the \emph{symmetric inverse monoid} whose identity is the identity map. We will also consider the set $\mathcal{I}_{m,n}$ of partial permutations from the set $\mathbf{m}$ to the set $\mathbf{n}$ for all $m,n \in \N$, where if $m = n$ we write $\mathcal{I}_n = \mathcal{I}_{n,n}$. These partial permutations will be used to represent the relative positions of conserved regions that appear in two circular bacterial genomes, and to represent inversion/deletion operations.

The \emph{symmetric inverse category}, denoted $\mathcal{I}$,  is the (small) category whose objects are the natural numbers and where the set of arrows from $m$ to $n$ is $\mathcal{I}_{m,n}$. For partial permutations $f \in \mathcal{I}_{m,n}$ and $g \in \mathcal{I}_{n,p}$, their composition $fg \in \mathcal{I}_{m,p}$ is such that, for all $i \in \dom(f)$, if $(i)f \in \im(f)\cap\dom(g)$ then $(i)fg = \big((i)f\big)g$ and if $(i)f \not \in \im(f)\cap\dom(g)$ then $i \not \in \dom(fg)$. 

The \emph{diagram} of $f \in \mathcal{I}_{m,n}$ is formed by arranging $m$ vertices labelled by elements of $\{1, \dots, m\}$ above $n$ vertices labelled by elements of $\{1, \dots, n\}$ forming two parallel rows of vertices. If $(i)f = j$ then there is an edge connecting $i$ in the upper row with $j$ in the lower row of the diagram (as in Figure \ref{fig:ppermdiagram}). 

\begin{figure}[h]
\begin{center}
\begin{tikzpicture}[scale = 0.5, thick, every node/.style={scale=0.75}]

\foreach \i in {1,...,4}
{
\node[circle,fill=black,inner sep=1pt,minimum size=3pt, label = below: $\i$] (l\i) at (\i,0) {};
}

\foreach \i in {1,...,5}
{
\node[circle,fill=black,inner sep=1pt,minimum size=3pt, label = above: $\i$] (u\i) at (\i,2) {};
}

\draw (2,2) -- (4,0);
\draw (4,2) -- (2,0);
\draw (5,2) -- (3,0);

\end{tikzpicture}
\caption[a]{A partial permutation $f$ in $\mathcal{I}_{5,4}$ with $\dom(f) = \{2,4,5\}$ and $\im(f) = \{2,3,4\}$.}
\label{fig:ppermdiagram}
\end{center}
\end{figure}
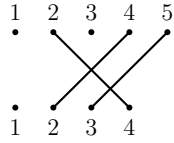

Using the diagrams of $f$ and $g$ it is often helpful to view their composition diagrammatically by first associating the vertices in the lower row of the diagram of $f$ with those in the upper row of the diagram of $g$, forming a graph called the product graph (see Figure \ref{fig:ppermprod}). If there is a path from $i$ in the upper row of the product graph to $j$ in the lower row then $(i)fg = j$.

\begin{figure}[h]
\begin{center}
\begin{tikzpicture}[scale = 0.5, thick, every node/.style={scale=0.75}]

\foreach \i in {1,...,5}
{
\node[circle,fill=black,inner sep=1pt,minimum size=3pt] (l\i) at (\i,0) {};
}

\foreach \i in {1,...,5}
{
\node[circle,fill=black,inner sep=1pt,minimum size=3pt, label = above: $\i$] (u\i) at (\i,2) {};
}

\draw (1,2) -- (4,0);
\draw (2,2) -- (2,0);
\draw (4,2) -- (3,0);
\draw (5,2) -- (5,0);

\node at (0,1) {$f=$};
\begin{scope}[yshift = -2cm]
\node at (0,1) {$g=$};

\foreach \i in {1,...,4}
{
\node[circle,fill=black,inner sep=1pt,minimum size=3pt, label = below: $\i$] (ll\i) at (\i,0) {};
}

\draw (1,2) -- (1,0);
\draw (2,2) -- (4,0);
\draw (3,2) -- (2,0);
\draw (5,2) -- (3,0);
\end{scope}

\begin{scope}[xshift = 6.5cm, yshift = -1cm]
\foreach \i in {1,...,4}
{
\node[circle,fill=black,inner sep=1pt,minimum size=3pt, label = below: $\i$] (l\i) at (\i,0) {};
}

\foreach \i in {1,...,5}
{
\node[circle,fill=black,inner sep=1pt,minimum size=3pt, label = above: $\i$] (u\i) at (\i,2) {};
}

\draw (2,2) -- (4,0);
\draw (4,2) -- (2,0);
\draw (5,2) -- (3,0);

\node at (6,1) {$=fg$};
\end{scope}
\draw[->] (5.5, 0) -- (7, 0);
\end{tikzpicture}
\caption{Calculating the product of partial permutations $f \in \mathcal{I}_{5,5}$ and $g \in \mathcal{I}_{5,4}$.}
\label{fig:ppermprod}
\end{center}
\end{figure}
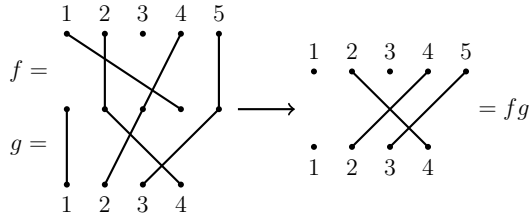

A partial permutation $f \in \mathcal{I}_{m,n}$ is said to be \emph{order preserving} if, for all $i,j \in \dom(f)$, we have $i < j$ if and only if $(i)f < (j)f$. Instances where $i < j$ but $(i)f > (j)f$ are called \emph{crossings}. The set of order preserving elements of $\mathcal{I}_{m,n}$ is denoted by $\mathcal{POI}_{m,n}$. A partial permutation $f \in \mathcal{I}_{m,n}$ with $\dom(f) = \{x_1, \dots, x_k\}$ is said to be \emph{orientation preserving} (cf. \citep{mcalister1998semigroups, catarino1999monoid}) if the sequence $\left((x_1)f, \dots, (x_k)f\right)$ is cyclic, in the sense that there exists at most one index $i \in \mathbf{k}$ such that $(x_i)f > \left(x_{i+1 \mod k}\right)f$. The set of orientation preserving elements of $\mathcal{I}_{m,n}$ is denoted $\mathcal{POPI}_{m,n}$. Order preserving partial permutations will arise when regions common to two genomes appear in the same order reading from position 1 to position $n$, while orientation preserving partial permutations will arise when these regions appear in the same (clockwise) cyclic order in both genomes. 

\section{An Algebraic Model of Inversions and Deletions}\label{sec:invdelmodel}

Given a reference frame of a genome $G$ specified by a bijection $\lambda: R \to \n$, inversions and deletions acting on $G$ are modelled by composing on the right of $\lambda$ by certain elements of the symmetric inverse category $\mathcal{I}$. For all $n \in \N^+$ let $s_{i;n}$ be the adjacent transposition $(i,i+1)$ in the symmetric group $S_n$ for all $1 \leq i \leq n-1$ and, to account for the circular nature of $G$, we also consider the 2-cycle $s_{n;n} = (1,n)$ since positions 1 and $n$ are adjacent in $G$. Letting
\[
\mathcal{T}_n = \{s_{i;n} : i \in \n\}, 
\]
composing on the right of $\lambda$ by elements of $\mathcal{T}_n$ will represent an inversion interchanging two adjacent regions in $G$. Note that the term ``inversion'' is used to refer to elements of $\mathcal{T}_n$ as well as the evolutionary operations they represent.

To model deletions, suppose $n \geq 2$ and let $d_{i;n}$ be the unique order preserving map in $\mathcal{POI}_{n, n-1}$ with $\dom(d_{i;n}) = \n \setminus \{i\}$ and $\im(d_{i;n}) = \{1, \dots, n-1\}$ (see Figure \ref{fig:deldiagram} for an example). 

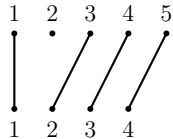
\begin{figure}[h]
\begin{center}
\begin{tikzpicture}[scale = 0.5, thick, every node/.style={scale=0.75}]

\foreach \i in {1,...,4}
{
\node[circle,fill=black,inner sep=1pt,minimum size=3pt, label= below:\i] (a) at (\i,0) {};
}
\foreach \i in {1,...,5}
{
\node[circle,fill=black,inner sep=1pt,minimum size=3pt, label= above:\i] (a) at (\i,2) {};
}

\draw (1,2) -- (1,0);
\draw (3,2) -- (2,0);
\draw (4,2) -- (3,0);
\draw (5,2) -- (4,0);
\end{tikzpicture}
\end{center}
\vspace*{-5mm}
\caption{The partial permutation diagram of $d_{2;5} \in \mathcal{I}_{5, 4}$. Note that this is still an injective map between the regions that are preserved, with the vertex corresponding to the position of the deleted region having degree 0.}
\label{fig:deldiagram}
\end{figure}

Letting 
\[
\mathcal{D}_n = \{d_{i;n} : i \in \n\},
\]
composing on the right of $\lambda$ by $d_{i;n} \in \mathcal{D}_n$ will represent deleting the region appearing in position $i$. Composing by a deletion yields a partial permutation from $R$ to $\n$, where a region $x$ is not in the domain if it has been deleted. Note that after we compose on the right by $d_{i;n}$, for all $j > i$ the region that appeared in position $j$ now appears in position $j-1$. For all $j < i$ the region appearing in position $j$ remains in that position. Figure \ref{fig:invedelcompexample} illustrates the corresponding compositions of deletions and inversions yielding the reference frame $g_2$ from the genome $A$ in Figure \ref{fig:example}.

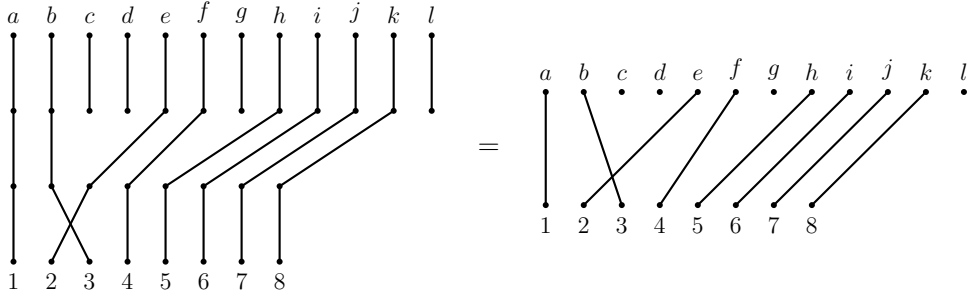
\begin{figure}[h]
\begin{center}
\begin{tikzpicture}[scale = 0.5, thick, every node/.style={scale=0.75}]
\foreach \i in {1,...,12}
{
\node[circle,fill=black,inner sep=1pt,minimum size=3pt] (a) at (\i,0) {};
}
\node[circle,fill=black,inner sep=1pt,minimum size=3pt, label= above:$a$] (a) at (1,2) {};
\node[circle,fill=black,inner sep=1pt,minimum size=3pt, label= above:$b$] (a) at (2,2) {};
\node[circle,fill=black,inner sep=1pt,minimum size=3pt, label= above:$c$] (a) at (3,2) {};
\node[circle,fill=black,inner sep=1pt,minimum size=3pt, label= above:$d$] (a) at (4,2) {};
\node[circle,fill=black,inner sep=1pt,minimum size=3pt, label= above:$e$] (a) at (5,2) {};
\node[circle,fill=black,inner sep=1pt,minimum size=3pt, label= above:$f$] (a) at (6,2) {};
\node[circle,fill=black,inner sep=1pt,minimum size=3pt, label= above:$g$] (a) at (7,2) {};
\node[circle,fill=black,inner sep=1pt,minimum size=3pt, label= above:$h$] (a) at (8,2) {};
\node[circle,fill=black,inner sep=1pt,minimum size=3pt, label= above:$i$] (a) at (9,2) {};
\node[circle,fill=black,inner sep=1pt,minimum size=3pt, label= above:$j$] (a) at (10,2) {};
\node[circle,fill=black,inner sep=1pt,minimum size=3pt, label= above:$k$] (a) at (11,2) {};
\node[circle,fill=black,inner sep=1pt,minimum size=3pt, label= above:$l$] (a) at (12,2) {};

\foreach \j in {1,...,12}
{
\draw (\j, 2) -- (\j, 0);
}

\foreach \i in {1,...,8}
{
\node[circle,fill=black,inner sep=1pt,minimum size=3pt] (a) at (\i,-2) {};
}

\foreach \i in {1,...,8}
{
\node[circle,fill=black,inner sep=1pt,minimum size=3pt, label=below:$\i$] (a) at (\i,-4) {};
}

\foreach \i in {1,2}
{
\draw(\i, 0) -- (\i, -2);
}
\foreach \i in {5,6}
{
\draw(\i, 0) -- (\i-2, -2);
}
\foreach \i in {8,9,10,11}
{
\draw(\i, 0) -- (\i-3, -2);
}
\foreach \i in {1,4,5,6,7,8}
{
\draw(\i, -2) -- (\i, -4);
}
\draw(2, -2) -- (3, -4);
\draw(3, -2) -- (2, -4);

\node at (13.5, -1) {\large{$=$}};

\begin{scope}[xshift = 14cm, yshift = -1.5cm]
\foreach \i in {1,...,8}
{
\node[circle,fill=black,inner sep=1pt,minimum size=3pt, label=below:$\i$] (a) at (\i,-1) {};
}
\node[circle,fill=black,inner sep=1pt,minimum size=3pt, label= above:$a$] (a) at (1,2) {};
\node[circle,fill=black,inner sep=1pt,minimum size=3pt, label= above:$b$] (a) at (2,2) {};
\node[circle,fill=black,inner sep=1pt,minimum size=3pt, label= above:$c$] (a) at (3,2) {};
\node[circle,fill=black,inner sep=1pt,minimum size=3pt, label= above:$d$] (a) at (4,2) {};
\node[circle,fill=black,inner sep=1pt,minimum size=3pt, label= above:$e$] (a) at (5,2) {};
\node[circle,fill=black,inner sep=1pt,minimum size=3pt, label= above:$f$] (a) at (6,2) {};
\node[circle,fill=black,inner sep=1pt,minimum size=3pt, label= above:$g$] (a) at (7,2) {};
\node[circle,fill=black,inner sep=1pt,minimum size=3pt, label= above:$h$] (a) at (8,2) {};
\node[circle,fill=black,inner sep=1pt,minimum size=3pt, label= above:$i$] (a) at (9,2) {};
\node[circle,fill=black,inner sep=1pt,minimum size=3pt, label= above:$j$] (a) at (10,2) {};
\node[circle,fill=black,inner sep=1pt,minimum size=3pt, label= above:$k$] (a) at (11,2) {};
\node[circle,fill=black,inner sep=1pt,minimum size=3pt, label= above:$l$] (a) at (12,2) {};

\draw(1,2) -- (1, -1);
\draw(2,2) -- (3, -1);
\draw(5, 2) -- (2, -1);
\draw(6,2) -- (4, -1);

\foreach \i in {8,9,10,11}
{
\draw(\i, 2) -- (\i-3, -1);
}
\end{scope}
\end{tikzpicture}
\end{center}
\caption{Given the reference frame of $A$ from Figure \ref{fig:example}, which is represented by the bijection where $a \mapsto 1$, $b \mapsto 2$ and so on, the deletions of the regions at positions 3,4,7 and 12 is given by composing on the right by the (non-unique) term $d_{12;12}d_{7;11}d_{4;10}d_{3;9}$. The subsequent inversion of positions 2 and 3 yielding $g_2$ is represented by composing on the right by $s_{2;8}$. After these operations $a$ is in position 1, $e$ is in position 2 and so on.}
\label{fig:invedelcompexample}
\end{figure}

Let $G_1$ and $G_2$ be arbitrary genomes and suppose that $G_2$ can be obtained from $G_1$ by inversions and deletions (note that we are not considering the most recent common ancestor of $G_1$ and $G_2$ here). For a fixed reference pair $(\lambda_{G_1}, \lambda_{G_2}) \in G_1 \times G_2$ a parsimonious inversion/deletion sequence transforming $\lambda_{G_1}$ into $\lambda_{G_2}$, when it exists, corresponds to a minimum length well-defined product $u$ of elements in 
\[
\mathcal{X} = \bigcup_{i \in \N^+} (\mathcal{D}_{i+1} \cup \mathcal{T}_{i})
\]
such that the bijection $\left(\lambda_{G_1}u\right)|_{\dom\left(\lambda_{G_1}u\right)}$ is equal to $\lambda_{G_2}$. Given a reference frame $\lambda_G$ of any genome $G$ and a well-defined product $u$ of elements in $\mathcal{X}$ we let 
\[
\left(\lambda_{G}u\right)|_{\dom\left(\lambda_{G}u\right)} = \overline{\lambda_{G}u}
\]
and let the length of $u$ be denoted by $\ell(u)$. For a fixed reference frame $\lambda_{G_1}$ of $G_1$ the quantity
\[
d(\lambda_{G_1}, G_2) = \min\left\{\ell(u) : \overline{\lambda_{G_1}u} \in G_2\right\}
\]
represents the length of a parsimonious inversion/deletion sequence transforming $G_1$ into $G_2$ beginning with the reference frame $\lambda_{G_1}$, while the quantity 
\[
d(G_1, G_2) = \min\left\{d(\lambda_{G_1}, G_2) : \lambda_{G_1} \in G_1\right\}
\]
represents the minimal inversion/deletion distance from $G_1$ to $G_2$. 

We now work towards establishing Lemma \ref{lem:del.first2} from which it follows, for a fixed reference frame $\lambda_{G_1}$ of $G_1$, that there exists a reference frame $\lambda_{G_2}$ of $G_2$ and a minimum length inversion/deletion sequence transforming $\lambda_{G_1}$ into $\lambda_{G_2}$ where the deletions occur first.

We proceed by first defining a digraph $\Delta$ whose paths represent the possible sequences of inversions, deletions, rotations and reflections of a genome that can occur. The digraph $\Delta$ (see Figure \ref{fig:localview}) has
\begin{itemize}
\item vertex set $\mathbb{N}$;
\item a directed edge from $n$ to $n$ for each element of $\mathcal{T}_n$ representing inversions for all $n \in \N^+$;
\item a directed edge from $n+1$ to $n$ for each element of $\mathcal{D}_n$ representing deletions for all $n \in \N^+$.
\end{itemize}
The digraph $\Delta$ also has a directed edge from $n$ to $n$ for all $n \in \N^+$ labelled by $c_n$ representing the $n$-cycle rotation $(1, \dots, n)$ in $S_n$, along with an edge labelled by $\alpha_n$ representing a reflection where 
\[
\alpha_n = \begin{dcases*}
(1,n)(2, n-1)\cdots(k,k+1) &  if $n = 2k$,\\
(1,n)(2,n-1)\cdots(k, k+2) & if $n= 2k+1$.
\end{dcases*}
\]
Note that the dihedral group $D_n$ is generated by $\{c_n, \alpha_n\}$. 

\begin{figure}[h]
\begin{center}
\begin{tikzpicture}[thick, every node/.style={scale=0.75}, yscale = 1.1, xscale = 1.3]

\node[circle, line width= 0.4pt, draw = black, inner sep = 2pt] (v1) at (0, 0) {$i-1$};
\node[circle, line width= 0.4pt, draw = black, inner sep = 2pt] (v2) at (2, 0) {\phantom{$i+1$}};
\node at (2, 0) {$i$};
\node[circle, line width= 0.4pt, draw = black, inner sep = 2pt] (v3) at (4, 0) {$i+1$};
\phantom{\node[circle, line width= 0.4pt, draw = black, inner sep = 2pt] (v4) at (-2, 0) {$i+1$};}
\phantom{\node[circle, line width= 0.4pt, draw = black, inner sep = 2pt] (v5) at (6, 0) {$i+1$};}
\node at (-2, 0){$\cdots$};
\node at (6, 0){$\cdots$};
\node at (3,0.075){$\vdots$};
\node at (1,0.075){$\vdots$};
\node at (4.025, 0.65){\small{$\cdots$}};
\node at (2.025, 0.65){\small{$\cdots$}};
\node at (0.025, 0.65){\small{$\cdots$}};

\draw[->, >=stealth](v3) to[bend right = 20] node[midway, above]{\small{$d_{1;i+1}$}} (v2);
\draw[->, >=stealth](v3) to[bend left = 20] node[midway, below]{\small{$d_{i+1;i+1}$}} (v2);
\draw[->, >=stealth](v2) to[bend right = 20] node[midway, above]{\small{$d_{1;i}$}} (v1);
\draw[->, >=stealth](v2) to[bend left = 20] node[midway, below]{\small{$d_{i;i}$}} (v1);
\draw[->, >=stealth](v1) to[bend right = 20] node[midway, above]{\small{$d_{1;i-1}$}} (v4);
\draw[->, >=stealth](v1) to[bend left = 20] node[midway, below]{\small{$d_{i-1;i-1}$}} (v4);
\draw[->, >=stealth](v5) to[bend right = 20] node[midway, above]{\small{$d_{1;i+2}$}} (v3);
\draw[->, >=stealth](v5) to[bend left = 20] node[midway, below]{\small{$d_{i+2;i+2}$}} (v3);

\draw[->, >=stealth](v2) to[out = 135, in = 105, looseness = 10]  node[midway, above, xshift = 0.1cm]{\small{$s_{1;i}$}} (v2);
\draw[->, >=stealth](v2) to[out = 45, in = 75, looseness = 10]  node[midway, above, xshift = 0.1cm]{\small{$s_{i;i}$}} (v2);
\draw[->, >=stealth](v2) to[out = -45, in = -75, looseness = 10]  node[midway, below, xshift = 0.1cm]{\small{$c_i$}} (v2);
\draw[->, >=stealth](v2) to[out = -135, in = -105, looseness = 10]  node[midway, below, xshift = 0.1cm]{\small{$\alpha_i$}} (v2);

\draw[->, >=stealth](v3) to[out = 135, in = 105, looseness = 10]  node[midway, above, xshift = 0.1cm]
{\small{$s_{1;i+1}$}} (v3);
\draw[->, >=stealth](v3) to[out = 45, in = 75, looseness = 10]  node[midway, above, xshift = 0.1cm]{\small{$s_{i+1;i+1}$}} (v3);
\draw[->, >=stealth](v3) to[out = -45, in = -75, looseness = 10]  node[midway, below, xshift = 0.1cm]{\small{$c_{i+1}$}} (v3);
\draw[->, >=stealth](v3) to[out = -135, in = -105, looseness = 10]  node[midway, below, xshift = 0.1cm]{\small{$\alpha_{i+1}$}} (v3);

\draw[->, >=stealth](v1) to[out = 135, in = 105, looseness = 10]  node[midway, above, xshift = 0.1cm]
{\small{$s_{1;i-1}$}} (v1);
\draw[->, >=stealth](v1) to[out = 45, in = 75, looseness = 10]  node[midway, above, xshift = 0.1cm]{\small{$s_{i-1;i-1}$}} (v1);
\draw[->, >=stealth](v1) to[out = -45, in = -75, looseness = 10]  node[midway, below, xshift = 0.1cm]{\small{$c_{i-1}$}} (v1);
\draw[->, >=stealth](v1) to[out = -135, in = -105, looseness = 10]  node[midway, below, xshift = 0.1cm]{\small{$\alpha_{i-1}$}} (v1);
\end{tikzpicture}
\end{center}
\caption{A local view of the digraph $\Delta$.}
\label{fig:localview}
\end{figure}
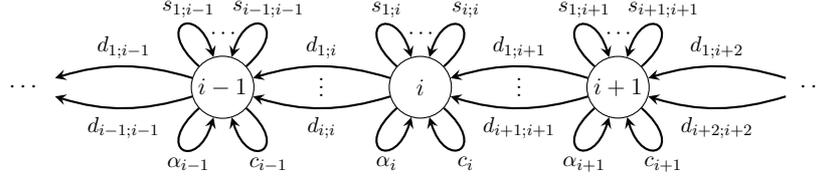

The \emph{free category} $\Delta^\ast$ on $\Delta$ contains all words  over the alphabet 
\[
\bigcup_{i \in \N^+} (\mathcal{D}_{i+1} \cup \mathcal{T}_{i} \cup \{c_i, \alpha_i\})
\]
corresponding to paths in $\Delta$ (note that edges may be traversed more than once if possible) that represent sequences of inversions, deletions, rotations and reflections. It can be verified (with the aid of diagrams as in Figure \ref{fig:relexamples} or using the presentation of the symmetric inverse category by \citet{east2020presentations}) that the following relations are satisfied by the corresponding partial permutations in $\mathcal{I}$ for all meaningful values of $n$, subject to stated constraints:
\begin{align*}
&\phantomseven d_{j;n}s_{n-1;n-1}               &&\hspace{-2cm}\text{if $i = n$ and $1 < j < n$}\ltag{(R1)}\\
&\phantomseven d_{1;n}c_{n-1}                   &&\hspace{-2cm}\text{if $i = j = n$}\ltag{(R2)}\\
&\phantomseven d_{n;n}c_{n-1}^{n-1}             &&\hspace{-2cm}\text{if $i = n$ and $j = 1$}\ltag{(R3)}\\
s_{i;n}d_{j;n} &= \caseseven d_{j;n}s_{i-1;n-1} &&\hspace{-2cm}\text{if $i > j$ and $i \leq n-1$}\ltag{(R4)}\\
&\phantomseven d_{j;n}s_{i;n-1}                 &&\hspace{-2cm} \text{if $i+1 < j$}\ltag{(R5)}\\
&\phantomseven d_{j+1;n}                        &&\hspace{-2cm}\text{if $i = j$ and $i \leq n-1$}\ltag{(R6)}\\
&\phantomseven d_{i;n}                          &&\hspace{-2cm}\text{if $i+1 = j$ and $i \leq n-1$} \ltag{(R7)}\\[2mm]
\lefttermtwo{c_ns_{i;n}} s_{i-1;n}c_n           &&\hspace{-2cm} \text{if $2 \leq i \leq n$}\ltag{(R8)}\\
&\phantomtwo s_{n;n}c_{n}                       &&\hspace{-2cm}\text{if $i = 1$} \ltag{(R9)}\\[2mm]
\lefttermtwoC{c_nd_{i;n}} d_{i-1;n}c_{n-1}      &&\hspace{-2cm} \text{if $2 \leq i \leq n$}\ltag{(R10)}\\
&\phantomtwoC d_{n;n}                           &&\hspace{-2cm}\text{if $i=1$,} \ltag{(R11)}\\[2mm]
\alpha_nd_{i;n} &=  d_{n-i+1;n}\alpha_{n-1}     &&\hspace{-2cm} \text{if $i \in \n$}\ltag{(R12)}\\[2mm]
\lefttermtwo{\alpha_ns_{i,n}}  s_{n-i;n}\alpha_n&&\hspace{-2cm} \text{if $1 \leq i \leq n-1$}\ltag{(R13)}\\
&\phantomtwo s_{n;n}\alpha_n                    &&\hspace{-2cm}\text{if $i = n$.} \ltag{(R14)}
\end{align*}

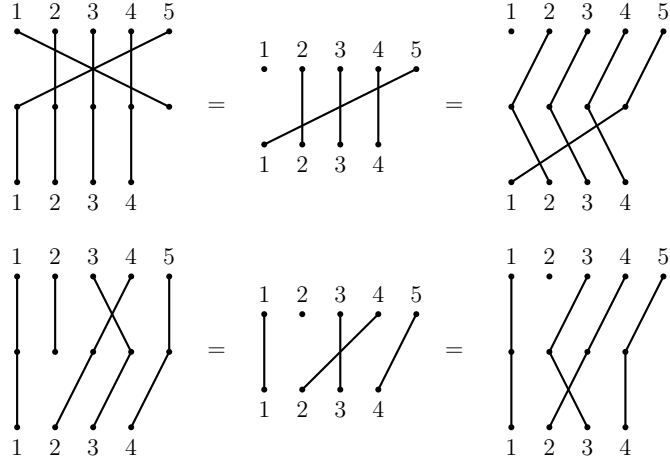
\begin{figure}[h]
\begin{center}
\begin{tikzpicture}[scale = 0.5, thick, every node/.style={scale=0.75}]

\foreach \i in {1,...,4}
{
\node[circle,fill=black,inner sep=1pt,minimum size=3pt, label= below:\i] (a) at (\i,0) {};
}
\foreach \i in {1,...,5}
{
\node[circle,fill=black,inner sep=1pt,minimum size=3pt] (a) at (\i,2) {};
\node[circle,fill=black,inner sep=1pt,minimum size=3pt, label= above:\i] (a) at (\i,4) {};
}
\draw (1,4) -- (1,2);
\draw (2,4) -- (2,2);
\draw (3,4) -- (4,2);
\draw (4,4) -- (3,2);
\draw (5,4) -- (5,2);

\draw (1,2) -- (1,0);
\draw (3,2) -- (2,0);
\draw (4,2) -- (3,0);
\draw (5,2) -- (4,0);

\node at (6.25, 2){$=$};

\begin{scope}[xshift = 6.5cm, yshift = 1cm]
\foreach \i in {1,...,4}
{
\node[circle,fill=black,inner sep=1pt,minimum size=3pt, label= below:\i] (a) at (\i,0) {};
}
\foreach \i in {1,...,5}
{
\node[circle,fill=black,inner sep=1pt,minimum size=3pt, label= above:\i] (a) at (\i,2) {};
}

\draw (1,2) -- (1,0);
\draw (3,2) -- (3,0);
\draw (4,2) -- (2,0);
\draw (5,2) -- (4,0);
\end{scope}

\node at (12.5, 2){$=$};

\begin{scope}[xshift = 13cm]
\foreach \i in {1,...,4}
{
\node[circle,fill=black,inner sep=1pt,minimum size=3pt, label= below:\i] (a) at (\i,0) {};
\node[circle,fill=black,inner sep=1pt,minimum size=3pt] (a) at (\i,2) {};
}
\foreach \i in {1,...,5}
{
\node[circle,fill=black,inner sep=1pt,minimum size=3pt, label= above:\i] (a) at (\i,4) {};
}
\draw (1,2) -- (1,0);
\draw (2,2) -- (3,0);
\draw (3,2) -- (2,0);
\draw (4,2) -- (4,0);

\draw (1,4) -- (1,2);
\draw (3,4) -- (2,2);
\draw (4,4) -- (3,2);
\draw (5,4) -- (4,2);
\end{scope}

\begin{scope}[yshift = 6.5cm]
\foreach \i in {1,...,4}
{
\node[circle,fill=black,inner sep=1pt,minimum size=3pt, label= below:\i] (a) at (\i,0) {};
}
\foreach \i in {1,...,5}
{
\node[circle,fill=black,inner sep=1pt,minimum size=3pt] (a) at (\i,2) {};
\node[circle,fill=black,inner sep=1pt,minimum size=3pt, label= above:\i] (a) at (\i,4) {};
}
\draw (1,4) -- (5,2);
\draw (2,4) -- (2,2);
\draw (3,4) -- (3,2);
\draw (4,4) -- (4,2);
\draw (5,4) -- (1,2);

\draw (1,2) -- (1,0);
\draw (2,2) -- (2,0);
\draw (3,2) -- (3,0);
\draw (4,2) -- (4,0);

\node at (6.25, 2){$=$};

\begin{scope}[xshift = 6.5cm, yshift = 1cm]
\foreach \i in {1,...,4}
{
\node[circle,fill=black,inner sep=1pt,minimum size=3pt, label= below:\i] (a) at (\i,0) {};
}
\foreach \i in {1,...,5}
{
\node[circle,fill=black,inner sep=1pt,minimum size=3pt, label= above:\i] (a) at (\i,2) {};
}

\draw (4,2) -- (4,0);
\draw (2,2) -- (2,0);
\draw (3,2) -- (3,0);
\draw (5,2) -- (1,0);
\end{scope}

\node at (12.5, 2){$=$};

\begin{scope}[xshift = 13cm]
\foreach \i in {1,...,4}
{
\node[circle,fill=black,inner sep=1pt,minimum size=3pt, label= below:\i] (a) at (\i,0) {};
\node[circle,fill=black,inner sep=1pt,minimum size=3pt] (a) at (\i,2) {};
}
\foreach \i in {1,...,5}
{
\node[circle,fill=black,inner sep=1pt,minimum size=3pt, label= above:\i] (a) at (\i,4) {};
}
\draw (1,2) -- (2,0);
\draw (2,2) -- (3,0);
\draw (3,2) -- (4,0);
\draw (4,2) -- (1,0);

\draw (2,4) -- (1,2);
\draw (3,4) -- (2,2);
\draw (4,4) -- (3,2);
\draw (5,4) -- (4,2);
\end{scope}

\end{scope}
\end{tikzpicture}
\end{center}
\vspace*{-5mm}
\caption{Diagrammatic illustration of the relation R2 (top row) with $s_{5;5}d_{5;5} = d_{1;5}c_{4}$ and R4 (bottom row) with $s_{3;5}d_{2;5} = d_{2;5}s_{2;4}$.}
\label{fig:relexamples}
\end{figure}

\begin{lem}\label{lem:del.first2}
Let $G_1$ be a circular genome with region set $R_1$ of size $m$ and let $G_2$ be a circular genome with region set $R_2$ of size $n$ where $R_2 \subset R_1$. Given a fixed reference frame $\lambda_1 : R_1 \to \mathbf{m}$ of $G_1$ suppose that $p$ is a minimum length product corresponding to a path in $\Delta^\ast$ such that $\overline{\lambda_1p} \in G_2$. There exists a reference frame $\lambda_2 : R_2 \to \mathbf{n}$ of $G_2$, a product $x$ consisting solely of deletions and a product $y$ consisting solely of inversions (both corresponding to paths in $\Delta^\ast$) such that $\ell(xy) = \ell(p)$ and $\overline{\lambda_1xy} = \lambda_2$.
\end{lem}

\begin{proof}
Let $p$ be a minimum length product in $\mathcal{I}$ corresponding to a word in $\Delta^\ast$  (which, by abuse of notation, we also denote by $p$) consisting of inversions and deletions such that $\overline{\lambda_1p} \in G_2$. Suppose also that $p$ contains at least one deletion. Using the relations in \ref{(R1)} -- \ref{(R14)} it is clear that $p$ is related to a word of the form $xyr$ where $x$ consists solely of deletions, $y$ consists solely of inversions and $r$ consists solely of dihedral symmetries. Since each application of these relations does not increase word length, it follows that $\ell(xy) \leq \ell(xyr) \leq \ell(p)$. 

Now, if $\overline{\lambda_1p}$ is in $G_2$ then so too is $\overline{\lambda_1xyr}$ since $p$ and $xyr$ evaluate to the same partial permutation in $\mathcal{I}$. As $r$ consists only of rotations and reflections, it then follows that $\overline{\lambda_1xyrr^{-1}} = \overline{\lambda_1xy}$ is also in $G_2$ as the partial permutation corresponding to $r^{-1}$ is a dihedral group element. The minimality of $\ell(p)$ together with the fact that $\ell(xy) \leq \ell(p)$ implies that $\ell(xy) = \ell(p)$ which completes the proof. 
\end{proof}

\begin{thm}\label{thm:del.first}
For a fixed reference frame $\lambda_{G_1}$ of $G_1$ there exists a product $u$ of elements in $\mathcal{X}$ minimising $d(\lambda_{G_1}, G_2)$ where the deletions occur first.
\end{thm}

\begin{proof}
This follows immediately from Lemma \ref{lem:del.first2}.
\end{proof}

\subsection{Reconstructing the Most Recent Common Ancestor}

Given genomes $G_1$ and $G_2$, candidates for their most recent common ancestor (under the parsimony criterion) are genomes $A$ with region set $R_1 \cup R_2$ minimising the sum $d(A, G_1) + d(A, G_2)$. While it could be the case that there are distinct reference frames $\lambda_{A_1}$ and $\lambda_{A_2}$ of $A$ such that $d(A, G_1) = d(\lambda_{A_1}, G_1)$ and $d(A, G_2) = d(\lambda_{A_2}, G_2)$ where $d(A, G_1) + d(A, G_2)$ is minimal, the following theorem establishes the fact that minimum length inversion/deletion sequences yielding $G_1$ and $G_2$ can always be thought of as beginning with a fixed reference frame of $A$. 

\begin{thm}\label{thm:common.ref}
Let $G_1$ and $G_2$ be genomes with region sets $R_1$ and $R_2$ respectively and suppose, among genomes with region set $R_1 \cup R_2$, that the genome $A$ has the property that $d(A, G_1) + d(A, G_2)$ is minimal. There exists a reference frame $\lambda_A$ of $A$ such that $d(A, G_1) = d(\lambda_A, G_1)$ and $d(A, G_2) = d(\lambda_A, G_2)$.  
\end{thm}

\begin{proof}
Suppose that $\overline{\lambda_{A_1}y} = \lambda_{G_1}$ and $\overline{\lambda_{A_2}z} = \lambda_{G_2}$ where $\lambda_{A_1}$ and $\lambda_{A_2}$ are reference frames of $A$ and where $y$ and $z$ are minimum length sequences of inversions and deletions. Suppose also that $w \in \{\alpha_n, c_n\}^\ast$ (that is, the set of all words whose letters are in $\{\alpha_n, c_n\}$) is such that $\lambda_{A_1}w = \lambda_{A_2}$. Since both $\alpha_n$ and $c_n$ are dihedral group elements there exists $\alpha_n^{-1}$ and $c_n^{-1}$ such that $\alpha_n\alpha_n^{-1}$ and $c_nc_n^{-1}$ are the identity map at $n \in \mathbb{N}$. As such, there exists a word $u \in \{\alpha_n, c_n\}^\ast$ such that $wu$ corresponds to the identity map at $n$ and so
\begin{equation}\label{eq:refframe1}
\lambda_{G_1} = \overline{\lambda_{A_1}y} = \overline{\lambda_{A_1}wuy} = \overline{\lambda_{A_2}uy}.
\end{equation}

Using the relations \ref{(R8)} -- \ref{(R14)} there exists a word $v \in \{\alpha_a, c_a\}^\ast$ for some $a \in \N^+$ and a word $y'$ of inversions and deletions such that $uy \sim y'v$ where $\ell(y') \leq \ell(y)$, in which case 
\begin{equation}\label{eq:comanc}
\overline{\lambda_{A_2}y'v} = \lambda_{G_1}
\end{equation}
by Equation \eqref{eq:refframe1}. Consider the fact that $\overline{\lambda_{A_2}y'vv^{-1}} =  \overline{\lambda_{A_2}y'}$ by Equation \eqref{eq:comanc}. Since multiplying on the right by $v^{-1} \in \{\alpha_a, c_a\}^\ast$ is equivalent to changing the reference frame of $\overline{\lambda_{A_2}y'v} = \lambda_{G_1}$, there is thus a sequence of inversions and deletions of length $\ell(y')$ such that $\overline{\lambda_{A_2}y'} \in G_1$ which completes the proof since $\ell(y') \leq \ell(y)$ and $y$ is minimal.
\end{proof}

To find the minimal distance $d(A, G_1) + d(A, G_2)$ given $G_1$ and $G_2$ we define a problem which we will refer to as the \emph{region alignment problem}, and show that if the solution to the region alignment problem is $k \in \mathbb{N}$ over all reference pairs in $G_1 \times G_2$ then these genomes have arisen minimally in $d(A, G_1) + d(A, G_2) = k + |R_1 \ominus R_2|$ inversions and deletions (where $\ominus$ denotes the symmetric difference of sets).

To define this problem, begin with $G_1$ and $G_2$ where $|R_1| = m$ and $|R_2| = n$. For a reference pair $(g_1, g_2) \in G_1 \times G_2$ where $g_1 = x_1\cdots x_m$ and $g_2 = y_1\cdots y_n$ construct a partial permutation $\sigma_{g_1, g_2} \in \mathcal{I}_{m,n}$ where $(i)\sigma_{g_1, g_2} = j$ if and only if $x_i = y_j$. Figure \ref{fig:genomepperm} illustrates an example of how $\sigma_{g_1, g_2}$ is formed.

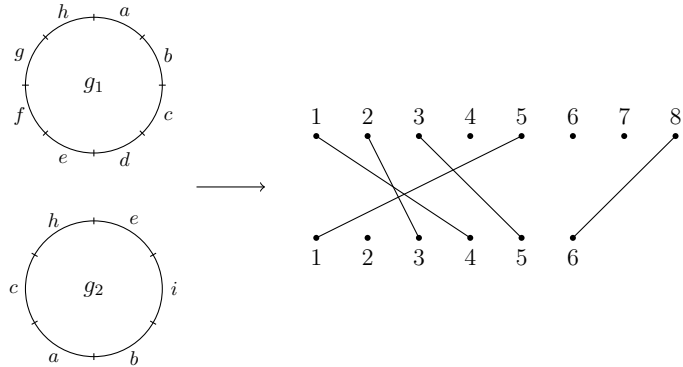
\begin{figure}[h]
\begin{center}
\begin{tikzpicture}[scale = 0.45, every node/.style={scale=0.75}]

      \draw (0,0) circle (2cm);
      \draw (0,0) node {$g_1$};
      \foreach \x in {0,45,90,...,315} \draw (\x:1.9cm) -- (\x:2.1cm);
      \def\rad{2.35cm}
      \draw (67.5:\rad) node[] {\footnotesize $a$}; 
      \draw (22.5:\rad) node[] {\footnotesize $b$};
      \draw (-22.5:\rad) node[] {\footnotesize $c$};  
      \draw (-67.5:\rad) node[] {\footnotesize $d$};
      \draw (-112.5:\rad) node[] {\footnotesize $e$};  
      \draw (-157.5:\rad) node[] {\footnotesize $f$};
      \draw (-202.5:\rad) node[] {\footnotesize $g$};  
      \draw (-247.5:\rad) node[] {\footnotesize $h$};  
      
\begin{scope}[yshift=-6cm]
      \draw (0,0) circle (2cm);
      \draw (0,0) node {$g_2$};
      \foreach \x in {30, 90, 150, 210, 270, 330} \draw (\x:1.9cm) -- (\x:2.1cm);
      \def\rad{2.35cm}
      \draw (0:\rad) node[] {\footnotesize $i$}; 
      \draw (60:\rad) node[] {\footnotesize $e$};
      \draw (120:\rad) node[] {\footnotesize $h$};  
      \draw (180:\rad) node[] {\footnotesize $c$};  
      \draw (240:\rad) node[] {\footnotesize $a$};
      \draw (300:\rad) node[] {\footnotesize $b$};  
      
\end{scope}

\draw[->](3, -3) -- (5, -3);

\begin{scope}[xshift = 5cm, yshift = -4.5cm, xscale =1.5, every node/.style={scale=0.75}]
\foreach \i in {1,...,6}
{
\node[circle,fill=black,inner sep=1pt,minimum size=3pt, label= below:\i] (a) at (\i,0) {};
}
\foreach \i in {1,...,8}
{
\node[circle,fill=black,inner sep=1pt,minimum size=3pt, label= above:\i] (a) at (\i,3) {};
}

\draw (5,3) -- (1,0);
\draw (1,3) -- (4,0);
\draw (2,3) -- (3,0);
\draw (3,3) -- (5,0);
\draw (8,3) -- (6,0);

\end{scope}
\end{tikzpicture}
\end{center}
\caption{Forming the partial permutation $\sigma_{g_1, g_2} \in \mathcal{I}_{8,6}$ given two reference frames $g_1 = abcdefgh$ and $g_2 = eibach$ of $G_1$ and $G_2$ with $R_1 = \{a,b,c,d,e,f,g,h\}$ and $R_2 = \{a,b,c,e,h,i\}$. The fact that $a$ appears at position $1$ in $g_1$ and at position $4$ in $g_2$ means $(1)\sigma_{g_1, g_2} = 4$.}
\label{fig:genomepperm}
\end{figure}

The elements in $\mathbf{m}\setminus \dom(\sigma_{g_1, g_2})$ and $\mathbf{n}\setminus \im(\sigma_{g_1, g_2})$ represent the regions in the symmetric difference $R_1\ominus R_2$ that do not appear in both genomes. The crossings of $\sigma_{g_1, g_2}$ represent the disorder of the labels in $R_1 \cap R_2$, in the sense that if $(i,j)$ is a crossing in $\sigma_{g_1, g_2}$ then the label $x_j \in R_1 \cap R_2$ appears before $x_i \in R_1 \cap R_2$ in the word $g_2$ while $x_j$ appears after $x_i$ in $g_1$. If $\sigma_{g_1, g_2}$ is order preserving then regions in $R_1 \cap R_2$ appear in the same order reading from position 1 to position $n$ around the genome.

Once $\sigma_{g_1, g_2}$ has been constructed multiplying on the left of $\sigma_{g_1, g_2}$ by an element of $\mathcal{T}_m$ represents an inversion acting on the reference frame $g_1$ while multiplying on the right by an element of $\mathcal{T}_n$ represents an inversion acting on $g_2$. In the region alignment problem we are given a reference pair $(g_1, g_2) \in G_1 \times G_2$ and ask for the minimum number of inversions acting on either $g_1$ or $g_2$ (or both) to place the regions in $R_1 \cap R_2$ in the same (clockwise) cyclic order in both genomes. The region alignment problem is stated mathematically as follows.

\begin{problem}\label{prob:main}
Let $G_1$ and $G_2$ be genomes with region sets $R_1$ and $R_2$ respectively where $|R_1| = m$ and $|R_2| = n$. For a reference pair $(g_1, g_2) \in G_1 \times G_2$, find a sequence $t_m$ of elements in $\mathcal{T}_m$ and a sequence $t_n$ of elements in $\mathcal{T}_n$ minimising $\ell(t_m) + \ell(t_n)$ such that $t_m\sigma_{g_1, g_2}t_n \in \mathcal{POPI}_{m,n}$. 
\end{problem}

This problem is a generalisation of the problem considered by \citet{egrinagy2014group} regarding the minimum inversion distance between two genomes with the same region set, which for a permutation $\sigma \in S_n$, asks for the minimum length sequence $t_n$ of elements in $\mathcal{T}_n$ such that $\sigma t_n$ is the identity. 

\begin{thm}\label{th:orientation}
If $G_1$ and $G_2$ are genomes with region sets $R_1$ and $R_2$ respectively and $\mu(g_1, g_2)$ is the minimum length solution to Problem \ref{prob:main} for a fixed reference pair $(g_1, g_2)$ then, under the parsimony criterion, $G_1$ and $G_2$ have descended from their most recent common ancestor in
\[
\ell(g_1, g_2) = |R_1 \ominus R_2| + \min\{\mu(g_1, g_2) : (g_1, g_2) \in G_1 \times G_2\}
\]
inversions and deletions.
\end{thm}

\begin{proof}
Let $k = \min\{\mu(g_1, g_2) : (g_1, g_2) \in G_1 \times G_2\}$. We begin by showing that $d(A, G_1)+ d(A, G_2)$ is bounded below by $k + |R_1 \ominus R_2|$ over all genomes $A$ with region set $R_1 \cup R_2$. To do this, suppose with the aim of obtaining a contradiction that there exists a reference pair $(g_1, g_2) \in G_1 \times G_2$ and products $t_m$ and $t_n$ of elements in $\mathcal{T}_m$ and $\mathcal{T}_n$ respectively with $t_m\sigma_{g_1, g_2}t_n \in \mathcal{POPI}_{m,n}$ (i.e. $\ell(t_m) + \ell(t_n) = k$), but where $G_1$ and $G_2$ have descended from their most recent common ancestor in strictly less than $k + |R_1 \ominus R_2|$ inversions and deletions.

By Theorem \ref{thm:del.first} there exists a genome $A$ with region set $R_1 \cup R_2$ minimising $d(A, G_1) + d(A, G_2)$ where $|R_1 \ominus R_2|$ deletions occur first. Further, by Theorem \ref{thm:common.ref} there exists a fixed reference frame $\lambda_{A}$ of $A$ where $d(A, G_1) = d(\lambda_{A}, G_1)$ and $d(A, G_2) = d(\lambda_{A}, G_2)$ in a minimal sum $d(A, G_1) + d(A, G_2)$. With these facts in mind and using Figure \ref{fig:example} as a guide, there exists a parsimonious inversion/deletion sequence yielding $G_1$ that proceeds by first deleting the regions in $R_2 \setminus R_1$ from a reference frame of $A$. This gives rise to a reference frame of an intermediate genome $G_1^\prime$. Likewise for $G_2$, the regions in $R_1 \setminus R_2$ are deleted first from $A$ to yield a reference frame of an intermediate genome $G_2^\prime$.  Since $G_1^\prime$ and $G_2^\prime$ have been obtained via deletions from the same reference frame of $A$, for all reference pairs $(g_1^\prime, g_2^\prime) \in G_1^\prime \times G_2^\prime$ the regions in $R_1 \cap R_2$ appear in the same clockwise cyclic order in both genomes. Thus, the partial permutation $\sigma_{g_1', g_2'}$ is orientation preserving (that is, $\sigma_{g_1', g_2'} \in \mathcal{POPI}_{m,n}$). Equivalently, there exists $(g_1^\prime, g_2^\prime) \in G_1^\prime \times G_2^\prime$ (possibly after rotating one of the genomes) such that $\sigma_{g_1^\prime, g_2^\prime}$ is order preserving (that is, $\sigma_{g_1', g_2'} \in \mathcal{POI}_{m,n}$). 

If $G_1$ and $G_2$ subsequently arise by sequences of inversions $p$ and $q$ in $\mathcal{T}_m$ and $\mathcal{T}_n$ acting on $g_1^\prime$ and $g_2^\prime$ respectively, then there exists $(g_1, g_2) \in G_1 \times G_2$ such that $p\sigma_{g_1^\prime, g_2^\prime}q = \sigma_{g_1, g_2}$. However, it would then follow that $p^{-1}\sigma_{g_1, g_2}q^{-1}$ is orientation preserving where $\ell(p^{-1}) + \ell(q^{-1}) = \ell(p) + \ell(q)$. Since we have $\sigma_{g_1^\prime, g_2^\prime} = p^{-1}\sigma_{g_1, g_2}q^{-1}$, the assumption that $k = \min\{\mu(g_1, g_2) : (g_1, g_2) \in G_1 \times G_2\}$ is contradicted if $\ell(p) + \ell(q) < \ell(t_m) + \ell(t_n)$. As such, $d(A, G_1)+ d(A, G_2)$ is bounded below by $k + |R_1 \ominus R_2|$ over all genomes $A$ with region set $R_1 \cup R_2$. 

To complete the proof, we show that if $k = \min\{\mu(g_1, g_2) : (g_1, g_2) \in G_1 \times G_2\}$ then there exists a genome $A$ with region set $R_1 \cup R_2$ such that $d(A, G_1)+ d(A, G_2) = k + |R_1 \ominus R_2|$. Beginning with a reference pair $(g_1, g_2) \in G_1 \times G_2$ with $\mu(g_1, g_2) = k$, suppose there exists sequences $t_m$ and $t_n$ of inversions from $\mathcal{T}_m$ and $\mathcal{T}_n$ respectively such that $t_m\sigma_{g_1, g_2}t_n \in \mathcal{POPI}_{m,n}$ (where $\ell(t_m) + \ell(t_n) = k$). Further, suppose that reference frames $g_1^\prime$ and $g_2^\prime$ of $G_1^\prime$ and $G_2^\prime$ are the result of these sequences of inversions acting on $g_1$ and $g_2$ respectively. By the circularity of the genomes (performing a rotation if necessary), it may be assumed without loss of generality that the regions in $R_1 \cap R_2 = \{r_1, \dots, r_h\}$ appear in the same order reading from 1 to $n$ in both $g_1^\prime$ and $g_2^\prime$.

Let $U_i$ be the set of regions appearing between $r_i$ and $r_{i+1}$ in $g_2^\prime$ reading from $1$ to $n$ for all $1 \leq i \leq h-1$, let $U_{h}$ be the set of regions appearing after $r_{h}$ (up to and including position $n$) in $g_2^\prime$ and let $U_0$ be the set of regions appearing before $r_1$ (from position 1 onward) in $g_2^\prime$. Beginning with $g_1^\prime$, form a genome $A$ with region set $R_1 \cup R_2$ by first inserting the regions in $U_0$ before $r_1$ in $g_1^\prime$ where the minimum element of $U_i$ is at position $1$ in $A$. If $U_0$ is empty, then $r_1$ is in position 1 in $A$. Next, for all $1 \leq i \leq h-1$ insert regions from $U_i$ into $g_1^\prime$ between $r_i$ and $r_{i+1}$ in any way that ensures the regions in $U_i$ appear in the same order that they do in $g_2^\prime$ reading from position $1$ to $n$. Finally, insert regions $U_{h}$ after $r_h$ in any way that ensures their appropriate order reading from $1$ to $n$ where the maximal element of $U_h$ is position $n$ in the resulting genome $A$. If $U_h$ is empty, then $r_h$ appears in position $n$ in $A$. Figure \ref{fig:insert} illustrates an example of these insertions. 

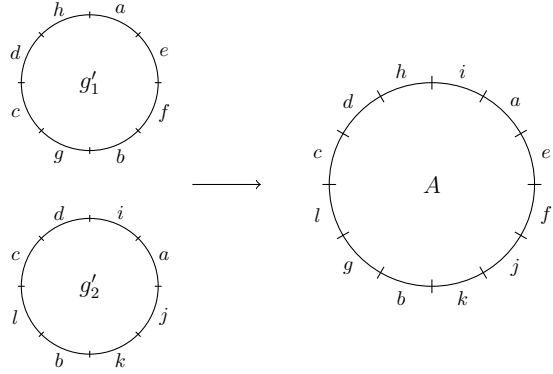
\begin{figure}[ht]
\begin{center}
\begin{tikzpicture}[scale = 0.45, every node/.style={scale=0.75}]

      \draw (0,0) circle (2cm);
      \draw (0,0) node {$g_1^\prime$};
      \foreach \x in {0,45,90,...,315} \draw (\x:1.9cm) -- (\x:2.1cm);
      \def\rad{2.35cm}
      \draw (67.5:\rad) node[] {\footnotesize $a$}; 
      \draw (22.5:\rad) node[] {\footnotesize $e$};
      \draw (-22.5:\rad) node[] {\footnotesize $f$};  
      \draw (-67.5:\rad) node[] {\footnotesize $b$};
      \draw (-112.5:\rad) node[] {\footnotesize $g$};  
      \draw (-157.5:\rad) node[] {\footnotesize $c$};
      \draw (-202.5:\rad) node[] {\footnotesize $d$};  
      \draw (-247.5:\rad) node[] {\footnotesize $h$};  
      
\begin{scope}[yshift=-6cm]
   \draw (0,0) circle (2cm);
      \draw (0,0) node {$g_2^\prime$};
      \foreach \x in {0,45,90,...,315} \draw (\x:1.9cm) -- (\x:2.1cm);
      \def\rad{2.35cm}
      \draw (67.5:\rad) node[] {\footnotesize $i$}; 
      \draw (22.5:\rad) node[] {\footnotesize $a$};
      \draw (-22.5:\rad) node[] {\footnotesize $j$};  
      \draw (-67.5:\rad) node[] {\footnotesize $k$};
      \draw (-112.5:\rad) node[] {\footnotesize $b$};  
      \draw (-157.5:\rad) node[] {\footnotesize $l$};
      \draw (-202.5:\rad) node[] {\footnotesize $c$};  
      \draw (-247.5:\rad) node[] {\footnotesize $d$};  
\end{scope}

\draw[->](3, -3) -- (5, -3);

\begin{scope}[scale = 2, xshift = 5cm, yshift = -1.5cm]
\draw (0,0) circle (1.5cm);
      \draw (0,0) node {$A$};
      \foreach \x in {0,30,60,...,330} \draw (\x:1.4cm) -- (\x:1.6cm);
      \def\rad{1.73cm}
      \draw (75:\rad) node[] {\footnotesize $i$};
      \draw (45:\rad) node[] {\footnotesize $a$};
      \draw (15:\rad) node[] {\footnotesize $e$};
      \draw (-15:\rad) node[] {\footnotesize $f$};
      \draw (-45:\rad) node[] {\footnotesize $j$};
      \draw (-75:\rad) node[] {\footnotesize $k$};
      \draw (-105:\rad) node[] {\footnotesize $b$};
      \draw (-135:\rad) node[] {\footnotesize $g$};
      \draw (-165:\rad) node[] {\footnotesize $l$};
      \draw (-195:\rad) node[] {\footnotesize $c$};
      \draw (-225:\rad) node[] {\footnotesize $d$};
      \draw (-255:\rad) node[] {\footnotesize $h$};
\end{scope}
\end{tikzpicture}
\end{center}
\caption{Given $g_1^\prime$ and $g_2^\prime$, we form the genome $A$ by inserting regions from $R_2 \setminus R_1$ into $g_1^\prime$ in their appropriate positions (with respect to elements of $R_1 \cap R_2$) and appropriate order from 1 to $n$.}
\label{fig:insert}
\end{figure}

Given the construction of $A$, it is easily verified that deleting the regions in $R_2 \setminus R_1$ from $A$ yields $g_1^\prime$ and that deleting the regions in $R_1 \setminus R_2$ from $A$ yields $g_2^\prime$. The inverses of the sequences $t_m$ and $t_n$ of inversions then act on $g_1^\prime$ and $g_2^\prime$ respectively to yield $g_1$ and $g_2$ in a total of $k + |R_1 \ominus R_2|$ inversions and deletions, as required.
\end{proof}

\section{Exact Algorithm and Complexity}\label{sec:exact}

We continue to assume that $G_1$ and $G_2$ are genomes with region sets $R_1$ and $R_2$,
respectively with $|R_1| = m$ and $|R_2| = n$. In this section we provide an
exact algorithm for computing sequences $t_m$ in $\mathcal{T}_m$ and $t_n$ in
$\mathcal{T}_n$ from Problem \ref{prob:main} such that $\ell(t_m) + \ell(t_n)$ is minimised and
$t_m \sigma_{g_1, g_2} t_n \in \mathcal{POPI}_{m, n}$. As previously, given $\sigma_{g_1, g_2}$
this minimised value is denoted by $\mu(\sigma_{g_1, g_2})$.
Additionally, we describe the asymptotic time and space complexity of the algorithm, and the limits of
its practical applicability on currently available computer hardware.

We denote the identity partial permutation on the set $X$ by
$\operatorname{id}_{X}$. For the purposes of this section, a \textit{graph}
$\Gamma$ is a triple $(V, E, X)$ where $V$ is a set whose elements are called the
\textit{vertices} of $\Gamma$; $X$ is the set of \textit{edge labels} of
$\Gamma$; and $E \subseteq V \times X \times V$ is the set of \textit{edges} of
$\Gamma$.

If $S$ is a semigroup and $X$ is a subset of $S$, then
we define the \textit{left Cayley graph} of $S$ with respect to $X$ to be the
graph with nodes $S$ and edges $(s, x, xs) \in S \times X \times S$ for all
$s\in S$ and for all $x\in X$; we denote this by $\Gamma_L(S, X)$. The \textit{right
Cayley graph} is defined dually, and is denoted $\Gamma_R(S, X)$.
If $\Gamma_L(S_1, X_1)$ and $\Gamma_R(S_2, X_2)$ are the left and right Cayley graphs (respectively) of semigroups $S_1$ and $S_2$ with respect to subsets $X_1$ and $X_2$ then, given $S_1 \subseteq S_2$, we define the \textit{union} of these graphs to be the graph with nodes $S_2$ and all of the edges belonging to $\Gamma_L(S_1,X_1)$ and $\Gamma_R(S_2, X_2)$. If $\Gamma = (V, E, X)$ is any graph and $A$ is a
subset of the vertices $V$ of $\Gamma$, then the \textit{subgraph} induced by $A$ is the graph $(A, E \cap (A \times X \times A), X)$.

Let $\sigma_{g_1, g_2} \in \mathcal{I}_{m,n}$ and suppose without loss of generality that $m \leq n$. We consider $\mathcal{I}_{m,m}$ and $\mathcal{POPI}_{m,n}$ to be embedded in $\mathcal{I}_{n,n}$ via an embedding $f$ where $(i)\sigma = j$ for $\sigma$ in $\mathcal{I}_{m,m}$ or $\mathcal{POPI}_{m,n}$ if and only if $(\sigma)f$ in $\mathcal{I}_{n,n}$ maps $i$ to $j$. The algorithm for determining $\mu(\sigma_{g_1, g_2})$ has the following steps:

\begin{enumerate}[label=(\roman*)]
  \item suppose that $\sigma_{g_1, g_2} \in \mathcal{I}_{m,n}$ where $|\dom(\sigma_{g_1, g_2})| = r$ and $m \leq n$; \label{item:i}
  \item let $X_j$ denote the generating set for $\mathcal{I}_{j,j}$ consisting of $\mathcal{T}_j$ and $\operatorname{id}_{\{1, \ldots, j - 1\}}$; \label{item:ii}
  \item compute the left  $\Gamma_L(\mathcal{I}_{m,m}, X_m)$ and right $\Gamma_R(\mathcal{I}_{n,n}, X_n)$ Cayley graphs of $\mathcal{I}_{m,m}$ and $\mathcal{I}_{n,n}$ with respect to the sets $X_m$ and $X_n$ respectively; \label{item:iii}
    \item compute the union $\Gamma_{m,n}$ of $\Gamma_L(\mathcal{I}_{m,m}, X_m)$ and $\Gamma_R(\mathcal{I}_{n,n}, X_n)$; \label{item:iv}
  \item compute the set $\mathscr{D}_r = \{\alpha \in \mathcal{I}_{n,n} : |\dom(\alpha)| = r\}$ in $\Gamma_{m,n}$. \label{item:v}
  \end{enumerate}
Given that the relation $\mathscr{D}$ on $\mathcal{I}_{n,n}$ where $\alpha \mathrel{\mathscr{D}} \beta$ if and only if $|\dom(\alpha)| = |\dom(\beta)|$ is an equivalence relation (called Green's $\mathscr{D}$-relation), the subgraph $\Delta_{n,r}$ induced by $\mathscr{D}_r$ is strongly connected (in the sense that there is a path in both directions between all pairs of vertices). Paths in this strongly connected component will traverse edges from $\Gamma_L(\mathcal{I}_{m,m}, X_m)$ representing inversions from $\mathcal{T}_m$ acting on the genome $G_1$ with $m$ regions, and edges from $\Gamma_R(\mathcal{I}_{n,n}, X_n)$ representing inversions from $\mathcal{T}_n$ acting on the genome $G_2$ with $n$ regions.
 \begin{enumerate}[label=(\roman*)]\setcounter{enumi}{5}
  \item compute the subgraph $\Delta_{n,r}$ of $\Gamma_{m,n}$ induced by $\mathscr{D}_r$ \label{item:vi};
\item $\mu(\sigma_{g_1, g_2})$ is then the minimum distance in $\Delta_{n,r}$ between $\sigma_{g_1, g_2}$ and any element of $\mathcal{POPI}_{m,n}$ in $\mathscr{D}_r$. \label{item:vii}
\end{enumerate}

Note that steps \ref{item:i} to \ref{item:vi} need only be computed once for each $m$,$n$ and $r$, and the resulting value of $\Delta_{n,r}$ can be memoised.

Steps \ref{item:i} and \ref{item:ii} have combined time complexity $\mathcal{O}(n)$; step \ref{item:iii}
has time and space complexity
\[
  \mathcal{O}(|X_n||\mathcal{I}_{n,n}|) = \mathcal{O}\left(n\ \sum_{r = 0}^{n} \binom{n}{r}^2
    r!\right)
\]
(using the Froidure-Pin Algorithm described by~\citet{Froidure1997aa} for
example). Hence the time and space complexity for this step is at best $\mathcal{O}(n!)$. Steps \ref{item:iv} and \ref{item:v} also have time complexity $\mathcal{O}(|X_n||\mathcal{I}_{n,n}|)$ since the number of vertices and edges in $\Gamma_L(\mathcal{I}_{m,m}, X_M)$ and $\Gamma_R(\mathcal{I}_{n,n}, X_n)$ is $ \mathcal{O}(|X_n||\mathcal{I}_{n,n}|)$. Hence steps \ref{item:i} to \ref{item:vi} overall have time and space complexity at best $\mathcal{O}(n!)$.

For step \ref{item:vii}, the distance between any two vertices in a graph can be found
in a number of ways. One approach would be to apply the Floyd-Warshall
Algorithm to compute the shortest path between every pair of vertices
in $\Delta_{n, r}$; the time complexity of Floyd-Warshall is $\mathcal{O}(n ^
3)$ where $n$ is the number of vertices in the graph.  Another approach is to
perform a depth or breadth first search. The version implemented
by~\citet{DeBeule2022aa} uses a breadth first search that also utilises the
automorphism group of the graph to avoid visiting multiple identical
branches. The automorphism groups of the graphs $\Delta_{n, r}$ are
non-trivial when $r \neq 0$ and this approach seems to offer the best
performance; see Table~\ref{table-aut-group-sizes}. Due to its high complexity the exact algorithm given above is only applicable for relatively small values of $n$; see Table~\ref{table-exact-times}.

\begin{table}[ht]
\centering
 \begin{tabular}{c|c|c|c|c|c|c|c}
   \diagbox{$r$}{$n$} & 2 & 3  & 4 & 5 & 6 & 7 & 8 \\
   \hline\hline
   0 & 1 & 1   & 1    & 1   & 1   & 1   & 1    \\
   1 & 8 & 72  & 384  & 200 & 288 & 392 & 512  \\
   2 & 2 & 144 & 1024 & 400 & 576 & 784 & 1024 \\
   3 & - & 72  & 128  & 200 & 288 & 392 & 512  \\
   4 & - & -   & 128  & 200 & 288 & 392 & 512  \\
   5 & - & -   & -    & 200 & 288 & 392 & 512  \\
   6 & - & -   & -    & -   & 288 & 392 & 512  \\
   7 & - & -   & -    & -   & -   & 392 & 512  \\
   8 & - & -   & -    & -   & -   & -   & 512  \\
 \end{tabular}
\vspace*{2.5mm}
\caption{Sizes of the automorphism groups of the graph $\Delta_{n, r}$.}
\label{table-aut-group-sizes}
\end{table}

\begin{table}[ht]
\centering
\resizebox{\textwidth}{!}{%
 \begin{tabular}{l|l|l|l|l|l|l}
   $n$ & 3  & 4 & 5 & 6 & 7 & 8 \\
   \hline\hline
   $|\mathcal{I}_{n, n}|$ & 34 & 209 & 1,546 & 13,327 & 130,922 & 1,441,729
   \\\hline
   \# of pairs of genomes & 153 & 2,704 & 69,225 & 2,503,836 &  122,783,857 & 7,859,043,648
   \\\hline
   Time for \ref{item:i}  to \ref{item:vi}  --- (s) & $3.880\times 10 ^ {-8}$  & $1.228\times 10 ^ {-6}$ & $7.227\times 10 ^ {-3}$ & $8.556 \times 10 ^ {-2}$ & $2.808 \times 10 ^ {0}$ & $2.595\times 10 ^ 2$ \\\hline
   Time for \ref{item:vii}  --- total (s) & $4.083 \times 10 ^ {-4}$ & $7.594 \times 10 ^
   {-3}$ & $1.938 \times 10 ^ {-1}$ & $7.402\times 10 ^ 0$ & $\sim 6$ minutes &
   $\sim 7$ hours  \\\hline
   Time for \ref{item:vii}  --- mean (s)  & $2.669 \times 10 ^ {-6}$ & $2.809\times 10 ^
   {-6}$ & $2.800\times 10 ^ {-6}$ & $2.956\times 10 ^ {-6}$ & ? & ?  \\
 \end{tabular}
 }
\vspace*{2.5mm}
 \caption{Time for the various steps in the exact algorithm when applied to
 every pair of genomes with $n$ and $k$ regions where $n \geq k$.}

 \label{table-exact-times}
\end{table}

To the best of the authors' knowledge, it is not clear whether there exists a polynomial time algorithm for Problem \ref{prob:main}. This problem is potentially a computationally difficult problem, and so from a practical perspective it appears that approximation based approaches, or variations, offer the most promise moving forward.  

To highlight this, we finish this section by showing that a variation of Problem \ref{prob:main} --- whether two genomes of equal size are an equivalent inversion/deletion distance from their most recent common ancestor --- is \textbf{NP}-complete. 
Since genomes of equal size arise from their common ancestor via the same number of deletions, 
by Theorem \ref{th:orientation} this problem is equivalent to a problem called \textsc{balancedsort}.  \textsc{Balancedsort} takes a partial permutation $\sigma \in \mathcal{I}_{m,n}$ and $k \in \mathbb{N}$, and asks whether there exist sequences $t_m$ and $t_n$ of inversions in $\mathcal{T}_m$ and $\mathcal{T}_n$ respectively with $\ell(t_m)+ \ell(t_n) \leq k$ such that $t_m\sigma t_n \in \mathcal{POI}_{m,n}$ and $\ell(t_m) = \ell(t_n)$. Note that we may consider $\mathcal{POI}_{m,n}$ instead of $\mathcal{POPI}_{m,n}$ here as, if there exists $(g_1, g_2) \in G_1 \times G_2$ such that $t_m\sigma_{g_1, g_2}t_n \in \mathcal{POPI}_{m,n}$, then there exists a reference pair $(h_1, h_2) \in G_1 \times G_2$ obtained by rotating at least one of the genomes such that $t_m\sigma_{h_1, h_2}t_n \in \mathcal{POI}_{m,n}$.

\begin{thm}\label{th:balancedsort}
Determining whether two bacterial genomes of equal size are an equivalent inversion/deletion distance from their most recent common ancestor is \textbf{NP}-complete.
\end{thm}

\begin{proof}
We proceed by showing that \textsc{balancedsort}, which is clearly in \textbf{NP}, is \textbf{NP}-complete. Consider an instance of the well known \textbf{NP}-complete problem \textsc{partition}, which consists of a multiset $A = \{a_1, \dots, a_n\}$ of positive integers and asks if there exists a partition of $A$ into disjoint multisets $X$ and $Y$ such that $\sum_{x\in X}x = \sum_{y\in Y}y$. Construct an instance of \textsc{balancedsort} from an instance of \textsc{partition} by letting $m = n + \sum_{i=1}^{n} a_i$ and by defining a partial permutation $\sigma \in \mathcal{I}_m$ to be such that 
\begin{itemize}
\item $(1)\sigma = a_1 + 1$ and $(a_1 + 1)\sigma = 1$, 
\item $\left(j + \displaystyle \sum_{i=1}^{j-1}a_i\right)\sigma = j+ \displaystyle \sum_{i=1}^{j}a_i$ for all $2 \leq j \leq n-1$ and 
\item $\left(j+ \displaystyle \sum_{i=1}^{j}a_i\right)\sigma = \left(j + \displaystyle \sum_{i=1}^{j-1}a_i\right)$ for all $2 \leq j \leq n-1$.
\end{itemize}
The value of $k$ is the sum of all elements in $A$. Figure \ref{fig:reduction} illustrates an example of this reduction, which is easily seen to run in polynomial time. 

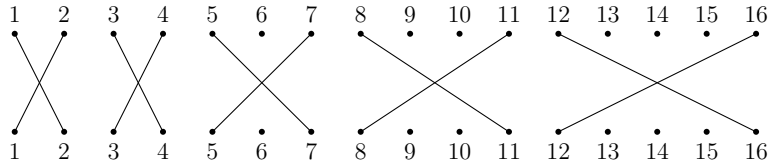
\begin{figure}[ht]
\begin{center}
\begin{tikzpicture}[scale = 0.65, every node/.style={scale=0.75}]

\foreach \i in {1,...,16}
{
\node[circle,fill=black,inner sep=1pt,minimum size=3pt, label= below:\i] (a) at (\i,0) {};
}
\foreach \i in {1,...,16}
{
\node[circle,fill=black,inner sep=1pt,minimum size=3pt, label= above:\i] (a) at (\i,2) {};
}

\draw (1,2) -- (2,0);
\draw (2,2) -- (1,0);
\draw (3,2) -- (4,0);
\draw (4,2) -- (3,0);
\draw (5,2) -- (7,0);
\draw (7,2) -- (5,0);
\draw (8,2) -- (11,0);
\draw (11,2) -- (8,0);
\draw (12,2) -- (16,0);
\draw (16,2) -- (12,0);

\end{tikzpicture}
\end{center}
\caption{For an instance $A = \{1,1,2,3,4\}$ of \textsc{partition}, an instance of \textsc{balancedsort} is constructed with $\sigma \in \mathcal{I}_{16}$ and $k = 11$.}
\label{fig:reduction}
\end{figure}

We first show that if there exists a partition of $A$ into $X$ and $Y$ such that $\sum_{x \in X}x = \sum_{y \in Y}y$ then there exists sequences $p$ and $q$ of inversions in $\mathcal{T}_m$ with $\ell(p) + \ell(q) \leq k$ such that $p\sigma q \in \mathcal{POI}_m$ where $\ell(p) = \ell(q)$. Define, without loss of generality, sequences $C_{a_j}$ of inversions (represented by 2-cycles) for all $a_j \in A$ where 
\[
C_{a_j} = \left(j + \sum_{i=1}^{j-1}a_{i}, 1+ j + \sum_{i=1}^{j-1}a_{i}\right)\cdots \left(j-1 + \sum_{i=1}^{j}a_{i}, j + \sum_{i=1}^{j}a_{i} \right)
\]
and note that the sequence $C_{a_j}$ removes the crossing consisting of domain elements $j + \sum_{i=1}^{j-1}a_{i}$ and $j + \sum_{i=1}^{j}a_{i}$ in a minimal way by left or right multiplication (but not both) without creating additional crossings.  Given a partition of $A$ into $X = \{a_{x_1}, \dots, a_{x_b}\}$ and $Y = \{a_{y_1}, \dots, a_{y_c}\}$ there is thus sequences $p = C_{a_{x_1}}\cdots C_{a_{x_b}}$ and $q = C_{a_{y_1}}\cdots C_{a_{y_c}}$ of inversions in $\mathcal{T}_m$ (where $\ell(p) + \ell(q) = k$ by construction) such that $p\sigma q \in \mathcal{POI}_m$ and where $\ell(p) = \ell(q)$ since $\sum_{x \in X}x = \sum_{y \in Y}y$. 

Conversely, suppose that for the constructed instance of {\sc balancedsort} there exists sequences $p$ and $q$ of inversions in $\mathcal{T}_m$ with $\ell(p) + \ell(q) \leq k$ such that $p\sigma q \in \mathcal{POI}_m$ where $\ell(p) = \ell(q)$. By construction $\ell(p) + \ell(q) = k$ where $k = \sum_{a \in A} a$. The sequence $p$ can be written in the form $C_{a_{x_1}}\cdots C_{a_{x_b}}$ and the sequence $q$ can be written in the form $C_{a_{y_1}}\cdots C_{a_{y_c}}$ where the sets $\{x_1, \dots, x_b\}$ and $\{y_1, \dots, y_c\}$ are disjoint. This is because each crossing is removed minimally by exclusively left or right multiplication of inversions without creating additional crossings. In other words, these sequences determine a partition of $A$ into $X = \{a_{x_1}, \dots, a_{x_b}\}$ and $Y = \{a_{y_1}, \dots, a_{y_c}\}$ where $\sum_{x\in X} x = \sum_{y \in Y}y$ follows from the fact that $p$ and $q$ are such that $\ell(p) = \ell(q)$. 
\end{proof}

\section{Discussion} 
\label{s:discussion}

This paper has introduced an algebraic framework for modelling two genome rearrangements, inversion and deletion, that are known to occur through the same biological process, namely site-specific recombination.  This framework involves the use of the symmetric inverse monoid, and appears to be the first usage of this type of semigroup model in the study of genome rearrangements.  As such a first step, there are on the one hand clear limitations of the model presented, and on the other, clear opportunities for further development.  

The most significant limitation involves the scope of the allowable rearrangements.  While the model treats a genome as a circular sequence of preserved regions of DNA (a standard way to view genomes in the rearrangement literature), it only permits inversions of adjacent regions, and only permits deletions of a single region at a time.  These two simplifying restrictions make the algebra more manageable by restricting the generating sets of the monoids involved.  But they are also broadly consistent with each other, since the underlying biological argument behind restricting the length of DNA sequence inverted or deleted is the same, because both arise from the same mechanism.  As noted in the Introduction, traditional rearrangement models do not restrict the length of the inverted region, and those that incorporate deletion (such as~\cite{el2000genome}) allow any length to be deleted, and with equal probability. They also generally allow the opposite operation, insertion, which typically occurs via different biological mechanism and so the savings in the computational simplicity come at an arguable cost to biological faithfulness --- as indeed they do in the present paper.  

A natural extension to the model presented here would be to allow longer regions to be inverted and/or deleted, perhaps along the lines attempted in~\cite{bhatia2020path}, which allows longer inversions in a group-theoretic model, but imposes a cost by length.  Indeed, some results here, such as Theorem~\ref{thm:del.first}, apply regardless of the generating set for $S_n$, or the number of regions being deleted.

Other generalisations may become available as a direct result of the algebraic framework.  For instance,
the algebraic formalisation using the symmetric inverse monoid can be generalised further by using monoids and categories of binary relations or partial functions. The use of certain binary relations $\lambda: R \to \n$ (or partial functions $\n \to R$ using the convention of positions to regions) allows one to account for repeated region labels, where an ordered pair $(r, n)$ is in $\lambda$ if and only if the region $r$ appears in position $n$ in a sequences of genome regions. For instance, the sequence $r_1r_2r_1r_3$ of regions where $\{r_1, r_2, r_3\} \subseteq R$ would correspond to the relation $\{(r_1, 1), (r_2, 2), (r_1, 3), (r_3, 4)\}$. 

Given sets $\m$ and $\n$ where $m,n \in \mathbb{N}$, the set of relations $\{(x_1, y_1), \dots, (x_j, y_j)\}$ such that $\{x_1, \dots, x_j\} \subseteq \m$, $y_i \in \n$ for all $1 \leq i \leq j$ and $y_i \neq y_j$ when $i \neq j$ is denoted by $\widehat{\mathcal{PT}}_{m,n}$, while the set of (analogously defined) partial functions from $\m \to \n$ is denoted $\mathcal{PT}_{m,n}$. One can define a (small) category $\widehat{\mathcal{PT}}$ whose objects are the natural numbers, and where the set of arrows from $m$ to $n$ is the set $\widehat{\mathcal{PT}}_{m,n}$ under the composition of binary relations. Given a relation $\lambda$ from $R$ to $n$ described above, we can compose on the right by elements of $\widehat{\mathcal{PT}}$ to represent not only inversions and deletions (since $\widehat{\mathcal{PT}}$ contains $\mathcal{I}$), but also to represent duplications of regions. To model a duplication we multiply on the right by relations of the form $\uV_{i;n} \in \widehat{\mathcal{PT}}_{n, n+1}$ where, without loss of generality (as in Figure \ref{fig:dupdiagram}), we have
\[
\uV_{i;n} = \{(1, 1), \dots, (i, i), (i, i+1), (i+1, i+2), \dots, (n, n+1)\}.
\]

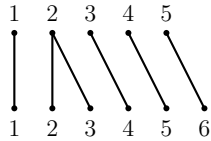
\begin{figure}[h]
\begin{center}
\begin{tikzpicture}[scale = 0.5, thick, every node/.style={scale=0.75}]

\foreach \i in {1,...,6}
{
\node[circle,fill=black,inner sep=1pt,minimum size=3pt, label= below:\i] (a) at (\i,0) {};
}
\foreach \i in {1,...,5}
{
\node[circle,fill=black,inner sep=1pt,minimum size=3pt, label= above:\i] (a) at (\i,2) {};
}

\draw (1,2) -- (1,0);
\draw (2,2) -- (2,0);
\draw (2,2) -- (3,0);
\draw (3,2) -- (4,0);
\draw (4,2) -- (5,0);
\draw (5,2) -- (6,0);
\end{tikzpicture}
\end{center}
\vspace*{-5mm}
\caption[a]{The relation diagram of $\uV_{2;5} \in \mathcal{\widehat{PT}}_{5, 6}$. Note that unlike the deletion in Figure \ref{fig:deldiagram} where the degree of the vertex labelled by 2 in the upper row was 0 (to represent the fact the region in position 2 was deleted), the upper row vertex labelled by 2 in this instance has degree 2 to represent the fact that the region in position 2 has been duplicated.}
\label{fig:dupdiagram}
\end{figure}

With this algebraic framework in mind, it is possible to consider the new problem of constructing the most recent common ancestor of two bacterial genomes (which may have repeated regions) under the three operations of inversions, deletions and duplications. Since the problem of reconstructing the most recent common ancestor of two genomes under exclusively inversions and deletions is a special case of this new problem, the same asymmetry present in the inversion/deletion model is also present in the inversion/deletion/duplication model given that only pre-existing genome regions may be duplicated. It is then natural to investigate whether similar combinatorial optimization problems regarding elements of $\widehat{\mathcal{PT}}$ have analogous interpretations to those presented here, such as Problem~\ref{prob:main}.

Finally, it would be interesting to explore whether the framework developed here could be cast in the representation-theoretic framework designed for maximum likelihood estimates for genome rearrangement models~\citep{serdoz2017maximum}, that is presented in~\citet{sumner2017representation, terauds2022new}.  Indeed, on one hand, \citet{terauds2022new} remark that it may be generalised to models using semigroups, while on the other hand, the representation theory of finite monoids including that of the symmetric inverse monoid has been well studied \citep{steinberg2016representation, Munn1964, solomon2002representations}.

\section{Declarations}
Andrew Francis was partially supported by Australian Research Council Discovery Project DP180102215. Data sharing is not applicable to this article as no datasets were generated or analysed. The authors have no competing interests to declare that are relevant to the content of this article.

\end{document}